\documentclass{amsart}
\usepackage{graphicx}
\usepackage{amssymb}
\usepackage{amsmath}
\usepackage{amsthm}
\usepackage{subfigure}

\newtheorem{thm}{Theorem}[section]

\newtheorem{lem}[thm]{Lemma}
\newtheorem{prop}[thm]{Proposition}
\theoremstyle{definition}
\newtheorem{defn}[thm]{Definition}
\newtheorem{rem}[thm]{Remark}

\newtheorem*{defn*}{Definition}
\newtheorem*{rems*}{Remarks}
\newtheorem*{rem*}{Remark}

\numberwithin{equation}{section}

\DeclareMathOperator {\Symp} {Symp}

\DeclareMathOperator {\Span} {span}

\def \algrest {\left [\Symp (\mathbb R^{2n})\right ]_{N}}

\def \algrestall {\bigl [\Lambda ^2(\mathbb R^{2n})\bigr ]_N}

\def \algrestclosed {\bigl [Z^2(\mathbb R^{2n})\bigr ]_N}

\def \local-algebra {\Lambda ^0(\mathbb R^2)/(\nabla H)}

\begin{document}

\title[Symplectic $S_{\mu}$ singularities] {Symplectic $S_{\mu}$ singularities}
\author{Wojciech Domitrz}
\address{Warsaw University of Technology\\
Faculty of Mathematics and Information Science\\
Plac Politechniki 1\\
00-661 Warsaw\\
Poland\\}

\email{domitrz@mini.pw.edu.pl}

\author{\.{Z}aneta Tr\c{e}bska}
\address{Warsaw University of Technology\\
Faculty of Mathematics and Information Science\\
Plac Politechniki 1\\
00-661 Warsaw\\
Poland\\}

\email{ztrebska@mini.pw.edu.pl}

\thanks{The work of W. Domitrz was supported by Polish MNiSW grant no. N N201 397237
}

\subjclass{Primary 53D05. Secondary 14H20, 58K50, 58A10.}

\keywords{symplectic manifold, curves, local symplectic algebra,
algebraic restrictions, relative Darboux theorem, singularities}

\maketitle

\begin{abstract}
We study the local symplectic algebra of the $1$-dimensional
isolated complete intersection singularity of type $S_{\mu}$. We
use the method of algebraic restrictions to classify symplectic
$S_{\mu}$ singularities. We distinguish these symplectic
singularities by discrete symplectic invariants. We also give the
geometric description of them.
\end{abstract}

\section{Introduction}

In this paper we study the symplectic classification of the
$1$-dimensional complete intersection singularity of type
$S_{\mu}$ in the symplectic space $(\mathbb R^{2n},\omega)$. We
recall that $\omega$ is a symplectic form if $\omega$ is a smooth
nondegenerate closed 2-form, and $\Phi:\mathbb{R}^{2n}
\rightarrow\mathbb{R}^{2n}$ is a symplectomorphism if $\Phi$ is a
diffeomorphism and $\Phi ^* \omega=\omega$.

\begin{defn} \label{symplecto}
Let $N_1, N_2$ be germs of subsets of symplectic space $(\mathbb{R}^{2n}, \omega)$. $N_1, N_2$
are \textbf{symplectically equivalent} if there exists a symplectomorphism-germ $\Phi:(\mathbb{R}^{2n}, \omega) \rightarrow(\mathbb{R}^{2n}, \omega)$
such that $\Phi(N_1)=N_2$.
\end{defn}

\medskip

The problem of symplectic classification of singular curves was
introduced by V. I. Arnold in \cite{Ar1}. Arnold proved that the
$A_{2k}$ singularity of a planar curve (the orbit with respect to
standard $\mathcal A$-equivalence of parameterized curves) split
into exactly $2k+1$ symplectic singularities (orbits with respect
to symplectic equivalence of parameterized curves). He
distinguished different symplectic singularities by different
orders of tangency of the parameterized curve to the
\emph{nearest} smooth Lagrangian submanifold. Arnold posed a
problem of expressing these new symplectic invariants in terms of
the local algebra's interaction with the symplectic structure and
he proposed to call this interaction the {\bf local symplectic
algebra}.

In \cite{IJ1} G. Ishikawa and S. Janeczko classified symplectic singularities of curves in the
$2$-dimensional symplectic space. All simple curves in this classification are quasi-homogeneous. A symplectic form on a $2$-dimensional manifold
is a special case of a volume form on a smooth manifold. The generalization of results
in \cite{IJ1} to volume-preserving classification of singular varieties and maps  in arbitrary dimensions was obtained in \cite{DR}.
The orbit of action of all diffeomorphism-germs agrees with volume-preserving orbit or splits into two volume-preserving orbits
(in the case $\mathbb K=\mathbb R$) for germs which satisfy a special weak form of quasi-homogeneity e.g. the weak quasi-homogeneity of varieties is a
quasi-homogeneity with non-negative weights $w_i\ge0$ and $\sum_i w_i>0$.

Symplectic singularity is stably simple if it is simple and remains simple if the ambient symplectic space is symplectically embedded
(i.e. as a symplectic submanifold) into a larger symplectic space.
In \cite{K} P. A. Kolgushkin classified the stably simple symplectic singularities of parameterized curves
(in the $\mathbb C$-analytic category). All stably simple symplectic singularities of curves are quasi-homogeneous too.

In \cite{DJZ2} new symplectic invariants of singular
quasi-homogeneous  subsets of a symplectic space were explained by
the algebraic restrictions of the symplectic form to these
subsets. The algebraic restriction is an equivalence class of the
following relation on the space of differential $k$-forms:

Differential $k$-forms $\omega_1$ and $\omega_2$ have the same
{\bf algebraic restriction} to a subset $N$ if
$\omega_1-\omega_2=\alpha+d\beta$, where $\alpha$ is a $k$-form
vanishing on $N$ and $\beta$ is a $(k-1)$-form vanishing on $N$.

In \cite{DJZ2} the generalization of Darboux-Givental theorem (\cite{ArGi})
to germs of arbitrary subsets of the symplectic space was obtained. This result reduces
the problem of symplectic classification of germs of quasi-homo\-ge\-neous subsets to
the problem of classification of algebraic restrictions of symplectic
forms to these subsets. For non-quasi-homogeneous subsets there is one more cohomological invariant except the algebraic restriction (\cite{DJZ2},
\cite{DJZ1}). The dimension of the space of algebraic restrictions of closed $2$-forms to a $1$-dimensional quasi-homogeneous isolated complete
intersection singularity $C$ is equal to the multiplicity of $C$ (\cite{DJZ2}).
In \cite{D} it was proved that the space of algebraic restrictions
of closed
$2$-forms to a $1$-dimensional (singular) analytic variety is finite-dimensional.

In \cite{DJZ2} the method of algebraic restrictions was applied to
various classification problems in a symplectic space. In
particular the complete symplectic classification of the classical
$1$-dimensional $S_5$ singularity were obtained. Most of different
symplectic singularity classes were distinguished by new discrete
symplectic invariants: the index of isotropy and the symplectic
multiplicity.

In this paper we obtain the complete symplectic classification of
the classical isolated complete intersection singularity $S_{\mu}$
for $\mu>5$ using the method of algebraic restrictions (Theorem
\ref{s-main}). We calculate discrete symplectic invariants for
this classification (Theorems \ref{lagr-s} and \ref{klas_s}) and
we present geometric descriptions of symplectic orbits (Theorem
\ref{geom-cond-s}).

In \cite{DT} following ideas from \cite{Ar1} and \cite{D}  new
discrete symplectic invariants - the Lagrangian tangency orders
were introduced and used to distinguish symplectic
singularities of simple planar curves of type $A-D-E$, symplectic
$S_5$ and $T_7$ singularities.

In this paper using Lagrangian tangency orders we are able to give
detailed classification of $S_{\mu}$
 singularity for $\mu>5$ (Theorem \ref{lagr-s}) and to present a
geometric description of its symplectic orbits (Theorem
\ref{geom-cond-s}).

The paper is organized as follows. In section \ref{method} we recall the method of algebraic restrictions.
In section \ref{discrete} we present  discrete symplectic invariants.
Symplectic classification of $S_{\mu}$ singularity is studied in section \ref{S}.

\section{The method of algebraic restrictions}
\label{method}

In this section we present basic facts on the method of algebraic
restrictions. The proofs of all results of this section can be
found in \cite{DJZ2}.

  Given a germ of a non-singular manifold $M$
denote by $\Lambda ^p(M)$ the space of all germs at $0$ of
differential $p$-forms on $M$. Given a subset $N\subset M$
introduce the following subspaces of $\Lambda ^p(M)$:
$$\Lambda ^p_N(M) = \{\omega \in \Lambda ^p(M): \ \ \omega (x)=0 \ \text {for any} \ x\in N \};$$
$$\mathcal A^p_0(N, M) = \{\alpha  + d\beta : \ \ \alpha
\in \Lambda _N^p(M), \ \beta \in \Lambda _N^{p-1}(M).\}$$ The
relation $\omega (x)=0$ means that the $p$-form $\omega $
annihilates any $p$-tuple of vectors in $T_xM$, i.e. all
coefficients of $\omega $ in some (and then any) local coordinate
system vanish at the point $x$.

\smallskip

\begin{defn}
\label{main-def} Let $N$ be the germ of a subset of $M$ and let
$\omega \in \Lambda ^p(M)$. The {\bf algebraic restriction} of
$\omega $ to $N$ is the equivalence class of $\omega $ in $\Lambda
^p(M)$, where the equivalence is as follows: $\omega $ is
equivalent to $\widetilde \omega $ if $\omega - \widetilde \omega
\in \mathcal A^p_0(N, M)$.
\end{defn}

\noindent {\bf Notation}. The algebraic restriction of the germ of
a $p$-form $\omega $ on $M$ to the germ of a subset $N\subset M$
will be denoted by $[\omega ]_N$. Writing $[\omega ]_N=0$ (or
saying that $\omega $ has zero algebraic restriction to $N$) we
mean that $[\omega ]_N = [0]_N$, i.e. $\omega \in A^p_0(N, M)$.

\medskip



Let $M$ and $\widetilde M$ be non-singular equal-dimensional
manifolds and let $\Phi: \widetilde M\to M$ be a local
diffeomorphism. Let $N$ be a subset of $M$. It is clear that $\Phi
^*\mathcal A_0^p(N, M) = \mathcal A_0^p(\Phi ^{-1}(N), \widetilde
M)$. Therefore the action of the group of diffeomorphisms can be
defined as follows: $\Phi ^*([\omega ]_N) = [\Phi ^*\omega ]_{\Phi
^{-1}(N)},$ where $\omega $ is an arbitrary $p$-form on $M$.

\begin{defn}Two algebraic restrictions
$[\omega ]_N$ and $[\widetilde \omega ]_{\widetilde N}$ are called {\bf
diffeomorphic} if there exists the germ of a diffeomorphism $\Phi:
\widetilde M\to M$ such that $\Phi(\widetilde N)=N$ and  $\Phi ^*([\omega ]_N) =[\widetilde \omega ]_{\widetilde N}$.
\end{defn}

\smallskip

\begin{rem}
The above definition does not depend on the choice of $\omega$ and
$\widetilde \omega$ since a local diffeomorphism maps forms with
zero algebraic restriction to $N$ to forms with zero algebraic
restrictions to $\tilde N$. If $M=\widetilde M$ and $N = \widetilde N$ then the definition of
diffeomorphic algebraic restrictions reduces to the following one:
two algebraic restrictions $[\omega ]_N$ and $[\widetilde \omega
]_N$ are diffeomorphic if there exists a local symmetry $\Phi $ of
$N $ (i.e. a local diffeomorphism preserving $N$) such that $[\Phi
^*\omega ]_N = [\widetilde \omega ]_N$.
\end{rem}

\begin{defn}
A subset $N$ of $\mathbb R^m$ is quasi-homogeneous if there exists
a coordinate system $(x_1,\cdots,x_m)$ on $\mathbb R^m$ and
positive numbers $\lambda_1,\cdots,\lambda_n$ such that for any
point $(y_1,\cdots,y_m)\in \mathbb R^m$ and any $t\in \mathbb R$
if $(y_1,\cdots,y_m)$ belongs to $N$ then a point
$(t^{\lambda_1}y_1,\cdots,t^{\lambda_m}y_m)$ belongs to $N$.
\end{defn}
The method of algebraic restrictions applied to singular
quasi-homogeneous subsets is based on the following theorem.

\begin{thm}[Theorem A in \cite{DJZ2}] \label{thm A}
Let $N$ be the germ of a quasi-homogeneous subset of $\mathbb
R^{2n}$. Let $\omega _0, \omega _1$ be germs of symplectic forms
on $\mathbb R^{2n}$ with the same algebraic restriction to $N$.
There exists a local diffeomorphism $\Phi $ such that $\Phi (x) =
x$ for any $x\in N$ and $\Phi ^*\omega _1 = \omega _0$.

Two germs of quasi-homogeneous subsets $N_1, N_2$ of a fixed
symplectic space $(\mathbb R^{2n}, \omega )$ are symplectomorphic
if and only if the algebraic restrictions of the symplectic form
$\omega $ to $N_1$ and $N_2$ are diffeomorphic.

\end{thm}

\medskip

Theorem \ref{thm A} reduces the problem of symplectic
classification of germs of singular quasi-homogeneous subsets to
the problem of diffeomorphic classification of algebraic
restrictions of the germ of the symplectic form to the germs of
singular quasi-homogeneous subsets.

The geometric meaning of zero algebraic restriction is explained
by the following theorem.

\begin{thm}[Theorem {\bf B} in \cite{DJZ2}] \label{thm B}  {\it The germ of a quasi-homogeneous
 set $N$  of a symplectic space
$(\mathbb R^{2n}, \omega )$ is contained in a non-singular
Lagrangian submanifold if and only if the symplectic form $\omega
$ has zero algebraic restriction to $N$.}\
\end{thm}

\begin{prop}[Lemma 2.20 in \cite{DJZ2}]
\label{zero-at-zero} Let $N\subset \mathbb R^m$. Let $W\subseteq
T_0\mathbb R^m$ be the tangent space to some (and then any)
non-singular submanifold containing $N$ of minimal dimension
within such submanifolds. If $\omega $ is the germ of a $p$-form
with zero algebraic restriction to $N$ then $\omega \vert _W = 0$.
\end{prop}

The following result shows that the method of algebraic
restrictions is very powerful tool in symplectic classification of
singular curves.

\begin{thm}[Theorem 2 in \cite{D}]
\label{main-alg} Let $C$ be the germ of a $\mathbb K$-analytic
curve (for $\mathbb K=\mathbb R$ or $\mathbb K=\mathbb C$). Then
the space of algebraic restrictions of germs of closed $2$-forms
to $C$ is a finite dimensional vector space.
\end{thm}

By a {\bf $\mathbb K$-analytic curve} we understand a subset of
$\mathbb K^m$ which is locally diffeomorphic to a $1$-dimensional
(possibly singular) $\mathbb K$-analytic subvariety of $\mathbb
K^m$. Germs of $\mathbb C$-analytic parameterized curves can be
identified with germs of irreducible $\mathbb C$-analytic curves.

We now recall basic properties of
algebraic restrictions which are useful for a description of this
subset (\cite{DJZ2}).

First we can reduce the dimension of the manifold we consider due
to the following propositions.

If the germ of a set $N\subset \mathbb R^m$ is contained in a
non-singular submanifold $M\subset \mathbb R^m$ then the
classification of algebraic restrictions to $N$ of $p$-forms on
$\mathbb R^m$ reduces to the classification of algebraic
restrictions to $N$ of $p$-forms on $M$. At first note that the
algebraic restrictions $[\omega ]_N$ and $[\omega \vert
_{TM}]_{_N}$ can be identified:

\begin{prop}
\label{reduction} Let $N$ be the germ at $0$ of a subset of
$\mathbb R^m$ contained  in a non-singular submanifold $M\subset
\mathbb R^m$ and let $\omega _1, \omega _2$ be $p$-forms on
$\mathbb R^m$. Then $[\omega _1]_N = [\omega _2]_N$ if and only if
$\bigl[\omega _1\vert _{TM}\bigr]_N = \bigl[\omega _2 \vert
_{TM}\bigr]_N$.
\end{prop}

The following, less obvious statement, means that the {\it orbits}
 of the algebraic restrictions $[\omega ]_N$ and $[\omega \vert
_{TM}]_{_N}$ also can be identified.

\begin{prop}
\label{main-reduction} Let $N_1,N_2$ be germs of subsets of
$\mathbb R^m$ contained in equal-dimensional non-singular
submanifolds $M_1, M_2$ respectively. Let $\omega _1, \omega _2$
be two germs of $p$-forms. The algebraic restrictions $[\omega
_1]_{N_1}$ and $[\omega _2]_{N_2}$ are diffeomorphic if and only
if the algebraic restrictions $\bigl[\omega _1\vert
_{TM_1}\bigr]_{N_1}$ and $\bigl[\omega _2\vert
_{TM_2}\bigr]_{N_2}$ are diffeomorphic.
\end{prop}

To calculate the space of algebraic restrictions of $2$-forms we
will use the following obvious properties.

\begin{prop}\label{d-wedge}
If $\omega\in \mathcal A_0^k(N,\mathbb R^{2n})$ then $d\omega
\in \mathcal A_0^{k+1}(N,\mathbb R^{2n})$ and $\omega\wedge
\alpha \in \mathcal A_0^{k+p}(N,\mathbb R^{2n})$ for any
$p$-form $\alpha$ on $\mathbb R^{2n}$.
\end{prop}

The next step of our calculation is the description of the
subspace of algebraic restriction of closed $2$-forms. The
following proposition is very useful for this step.
\begin{prop}
\label{th-all-closed} Let $a_1,\dots , a_k$ be a basis of the
space of algebraic restrictions of $2$-forms  to $N$ satisfying
the following conditions
\begin{enumerate}
\item $da_1 = \cdots = da_j = 0$, \item the algebraic restrictions
$da_{j+1}, \dots , da_k$ are linearly independent.
\end{enumerate}
 Then $a_1, \dots , a_j$ is a basis of the space
of algebraic restriction of closed $2$-forms to $N$.
\end{prop}

Then we need to determine which algebraic restrictions of closed
$2$-forms are realizable by symplectic forms. This is possible due
to the following fact.

\begin{prop}\label{rank}
 Let $N\subset \mathbb R^{2n}$. Let $r$ be
the minimal dimension of non-singular submanifolds of $\mathbb
R^{2n}$ containing $N$. Let $M$ be one of such $r$-dimensional
submanifolds. The algebraic restriction $[\theta ]_N$ of the germ
of closed $2$-form $\theta $ is realizable by the germ of a
symplectic form on $\mathbb R^{2n}$ if and only if $rank (\theta
\vert _{T_0M})\ge 2r - 2n$.
\end{prop}

Let us fix the following notations:

\smallskip

\noindent $\bullet$  $\algrestall $: \ \ the vector space
consisting of algebraic restrictions of germs of all $2$-forms on
$\mathbb R^{2n}$ to the germ of a subset $N\subset \mathbb
R^{2n}$;

\smallskip

\noindent $\bullet$  $\algrestclosed$: \ \ the subspace of
$\algrestall $ consisting of algebraic restrictions of germs of
all closed $2$-forms on $\mathbb R^{2n}$ to $N$;

\smallskip

\noindent $\bullet$ $\algrest $: \ \ the open set in
$\algrestclosed$ consisting of algebraic restrictions of germs of
all symplectic $2$-forms on $\mathbb R^{2n}$ to $N$.

\medskip

\section{Discrete symplectic invariants.}\label{discrete}

We can use some discrete symplectic invariants to characterize simplectic singularity classes. They show how far is a curve $N$ from the closest non-singular Lagrangian submanifold.

 The first one is a symplectic
multiplicity (\cite{DJZ2}) introduced  in \cite{IJ1} as a
symplectic defect of a curve.

Let $N$ be a germ of a subset of $(\mathbb R^{2n},\omega)$.

\begin{defn}
\label{def-mu}
 The {\bf symplectic multiplicity} $\mu_{sympl}(N)$ of  $N$ is the codimension of
 a symplectic orbit of $N$ in an orbit of $N$ with respect to the action of the group of local diffeomorphisms.
\end{defn}

The second one is the index of isotropy \cite{DJZ2}.

\begin{defn}
The {\bf index of isotropy} $\iota(N)$ of $N$ is the maximal
order of vanishing of the $2$-forms $\omega \vert _{TM}$ over all
smooth submanifolds $M$ containing $N$.
\end{defn}

They can be described in terms of algebraic restrictions.

\begin{prop}[\cite{DJZ2}]\label{sm}
The symplectic multiplicity  of the germ of a quasi-homogeneous subset $N$ in a
symplectic space is equal to the codimension of the orbit of the
algebraic restriction $[\omega ]_N$ with respect to the group of
local diffeomorphisms preserving $N$  in the space of algebraic
restrictions of closed  $2$-forms to $N$.
\end{prop}

\begin{prop}[\cite{DJZ2}]\label{ii}
The index of isotropy  of the germ of a quasi-homogeneous subset $N$ in a
symplectic space $(\mathbb R^{2n}, \omega )$ is equal to the
maximal order of vanishing of closed $2$-forms representing the
algebraic restriction $[\omega ]_N$.
\end{prop}

One more discrete symplectic invariant were introduced in \cite{D}
following ideas from \cite{Ar1} which is defined specifically for
a parameterized curve. This is the maximal tangency order of a
curve $f:\mathbb R\rightarrow M$ to a smooth Lagrangian
submanifold. If $H_1=...=H_n=0$ define a smooth submanifold $L$ in
the symplectic space then the tangency order of a curve $f:\mathbb
R\rightarrow M$ to $L$ is the minimum of the orders of vanishing
at $0$ of functions $H_1\circ f,\cdots, H_n\circ f$. We denote the
tangency order of $f$ to $L$ by $t(f,L)$.

\begin{defn}
The {\bf Lagrangian tangency order} $Lt(f)$\textbf{ of a curve} $f$ is the
maximum of $t(f,L)$ over all smooth Lagrangian submanifolds $L$ of
the symplectic space.
\end{defn}

The Lagrangian tangency order of a quasi-homogeneous curve in a
symplectic space can also be  expressed in terms of algebraic
restrictions.

\begin{prop}[\cite{D}]\label{lto}
Let $f$ be the germ of a quasi-homogeneous curve such that the
algebraic restriction of a symplectic form to it can be
represented by a closed $2$-form vanishing at $0$. Then the
Lagrangian tangency order of the germ of a quasi-homogeneous curve
$f$ is the maximum of the order of vanishing on $f$ over all
$1$-forms $\alpha$ such that $[\omega]_f=[d\alpha]_f$
\end{prop}

In \cite{DT}  the above invariant were generalized for germs of
curves and multi-germs of curves which may be parameterized
analytically since Lagrangian tangency order is the same for every
'good' analytic parameterization of a curve.

Consider a multi-germ $(f_i)_{i\in\{1,\cdots,r\}}$ of analytically
parameterized curves $f_i$. For any smooth submanifold $L$ in the
symplectic space we have $r$-tuples $(t(f_1,L), \cdots,
t(f_r,L))$.

\begin{defn}
For any $I\subseteq \{1,\cdots, r\}$ we define \textbf{the tangency order of the multi-germ } $(f_i)_{i\in I}$ to $L$:
$$t[(f_i)_{i\in\ I},L]=\min_{i\in\ I} t(f_i,L).$$
\end{defn}

\begin{defn}
The {\bf Lagrangian tangency order} $Lt((f_i)_{i\in\ I})$ \textbf{of a multi-germ } $(f_i)_{i\in I}$ is the maximum of $t[(f_i)_{i\in\ I},L]$ over all smooth Lagrangian submanifolds $L$ of the symplectic space.
\end{defn}

 For multi-germs one can also define relative invariants according to selected branches or collections of branches \cite{DT}.



\begin{defn}
For fixed $j\in I$  the \textbf{ Lagrangian tangency order related to}  $f_j$ of a multi-germ  $(f_i)_{i\in I}$ \textbf{} denoted by $Lt[(f_i)_{i\in I}: f_j]$ is the maximum of $t[(f_i)_{i\in I\setminus\{j\}},L]$ over all smooth Lagrangian submanifolds $L$ of the symplectic space for which $t(f_j,L)=Lt(f_j)$.
\end{defn}

These invariants have geometric interpretation.
If $Lt(f_i)=\infty$ then a branch $f_i$ is included in a smooth Lagrangian submanifold.
If $Lt((f_i)_{i\in\ I})=\infty$ then exists a Lagrangian submanifold including all curves $f_i$ for $i\in I$.

We may use these invariants for distinguishing symplectic singularities.

\newpage

\section{Symplectic $S_{\mu}$-singularities}\label{S}

Denote by $(S_{\mu})$  (for $\mu>5$) the class of varieties in  a fixed symplectic
space $(\mathbb R^{2n}, \omega )$ which are diffeomorphic to
\begin{equation}
\label{defs} S_{\mu}=\{x\in \mathbb R ^{2n\geq 4}\,:x_1^2-x_2^2-x_3^{\mu -3}=x_2 x_3=x_{\geq 4}=0\}.\end{equation}

This is the classical 1-dimensional isolated complete intersection singularity.  Let $N\in (S_{\mu})$. Then $N$ is the union of two $1$-dimensional components invariant under action of local diffeomorphisms preserving $N$: $C_1$ -- diffeomorphic to $A_1$ singularity and $C_2$ -- diffeomorphic to $A_{\mu -4}$ singularity. N is quasi-homogeneous with weights $w(x_1)=w(x_2)=\mu-3,\, w(x_3)=2$ when $\mu$ is an even number, or $w(x_1)=w(x_2)=(\mu-3)/2,\ w(x_3)=1$ when $\mu$ is an odd number. In our paper we often use the notation $r=\mu-3$.

\medskip

 We will use the method of algebraic restrictions to obtain a complete classification of
symplectic singularities in $(S_{\mu})$ presented in the following theorem.

\begin{thm}\label{s-main}
Any stratified submanifold of the symplectic space $(\mathbb
R^{2n},\sum_{i=1}^n dp_i \wedge dq_i)$ which is diffeomorphic to
$S_{\mu}$ is symplectically equivalent to one and only one of the
normal forms $S_{\mu}^{i,j}$. The
parameters $c_i$ of the normal forms are moduli.

\smallskip

\noindent $S_{\mu}^0:$ \ $p_1^2 - p_2^2 - q_1^r = 0, \ \ p_2q_1 = 0, \ \
q_2 = c_1q_1 - c_2p_1, \ \ p_{\ge 3} = q_{\ge 3} = 0$;

\smallskip

\noindent $S_{\mu\;2}^{k} (1\leq k\leq\mu-5):$  $p_2^2-p_1^2-q_1^r=0, \ \ p_1q_1=0,  \
q_2=c_3p_1+\frac{c_{4+k}}{k+1}q_1^{k+1}$, \par  $p_{\ge 3}=q_{\ge 3}=0, \ c_{4+k}\ne 0;$

\smallskip

\noindent $S_{\mu\;\;2}^{\mu-4}:$  $p_2^2-p_1^2-q_1^r=0, \ \ p_1q_1=0,  \
q_2=c_3p_1+\frac{c_{\mu}}{r}q_1^{r}, \  p_{\ge 3}=q_{\ge 3}=0, \; c_3c_{\mu}=0;$

\smallskip

\noindent $S_{\mu\;\;r}^{1+k} (1\leq k\leq\mu-6):$  $p_1^2-q_1^2-q_2^r=0, \ \ q_1q_2=0,  \ p_2=p_1q_2^{k}(c_{4+k}+c_{5+k}q_2)$, \par  $p_{\ge 3}=q_{\ge 3}=0,$
\ $\; \; \; \; c_{4+k}\ne 0$;

\smallskip

\noindent $S_{\mu\;\;r}^{\mu-4}:$  $p_1^2-q_1^2-q_2^r=0, \ \ q_1q_2=0,  \ p_2=c_{\mu-1}p_1q_2^{r-2}, \  p_{\ge 3}=q_{\ge 3}=0;$

\smallskip

\noindent $S_{\mu}^{3,1}:$  $p_1^2-p_2^2-p_3^r=0, \  p_2p_3=0,  \
q_1=\frac{1}{2}p_3^{2},\ q_2=-c_4p_1p_3,  \  p_{\ge 4}=q_{\ge 3}=0$;

\smallskip

\noindent $S_{\mu}^{2+k,1} (2\leq k\leq\mu-4):$  $p_1^2-p_2^2-p_3^r=0, \  p_2p_3= 0,  \
q_1=\frac{c_{4+k}}{k+1}p_3^{k+1}$, \ $q_2=-p_1p_3,$ \par  $ \  p_{\ge 4}=q_{\ge 3}=0$, \ ($c_{4+k}\ne 0$  for $ 2\leq k\leq\mu-5);$

\smallskip

\noindent $S_{\mu}^{3+k,k}(2\leq k\leq\mu-4):$  $p_1^2-p_2^2-p_3^r=0, \  p_2p_3=0,  \
q_1=\frac{1}{k+1}p_3^{k+1},  \  p_{\ge 4}=q_{\ge 2}=0;$

\smallskip

\noindent $S_{\mu}^{\mu}:$  $p_1^2-p_2^2-p_3^r=0, \  p_2p_3=0,  \
\  p_{\ge 4}=q_{\ge 1}=0,$

\smallskip

(by $r$ we denote $\mu-3$).

\end{thm}

 In section  \ref{s-class}  we calculate
the manifolds $[{\rm Symp} (\mathbb R^{2n})]_{S_{\mu}}$ and classify its algebraic restrictions. This allows us to decompose $S_{\mu}$ into symplectic singularity classes. In section  \ref{s-normal}  we transfer the normal forms for algebraic restrictions to symplectic normal forms to obtain the
proof of Theorem \ref{s-main}. In section \ref{s-lagr} we use
Lagrangian tangency orders to distinguish more symplectic
singularity classes. In section \ref{s-geom_cond} we propose a
geometric description of these singularities which confirms this
more detailed classification. Some of the proofs are presented in
section \ref{s-proof}.

\subsection{Algebraic restrictions and their classification}\label{s-class}

One has the relations for $(S_{\mu})$-singularites
\begin{equation}
[d(x_2 x_3)]_{S_{\mu}}=[x_2 dx_3+x_3 dx_2]_{S_{\mu}}=0
\label{s1}
\end{equation}
\begin{equation}
[d(x_1^2-x_2^2-x_3^{\mu-3})]_{S_{\mu}}=[2x_1dx_1-2x_2dx_2-(\mu-3)x_3^{\mu-4}dx_3]_{S_{\mu}}=0
\label{s2}
\end{equation}
Multiplying these relations by suitable $1$-forms we obtain the relations in Table \ref{tabs1}.

\renewcommand*{\arraystretch}{1.3}
\begin{small}
\begin{table}[h]
\begin{center}
\begin{tabular}{|c|c|c|}

 \hline

        & relations & proof\\ \hline

   1. & $[x_2dx_2 \wedge dx_3]_N=0$ & (\ref{s1})$\wedge\, dx_2$ \\ \hline

   2. &  $[x_3dx_2 \wedge dx_3]_N=0$ & (\ref{s1})$\wedge\, dx_3$\\ \hline

   3. &  $[x_1dx_1 \wedge dx_2]_N=0$  & (\ref{s2})$\wedge\, dx_2\;$ and row 2. \\\hline

    4. & $[x_1dx_1 \wedge dx_3]_N=0$  & (\ref{s2})$\wedge\, dx_3\;$ and row 1. \\\hline

    5. &  $[x_3dx_1 \wedge dx_2]_N=[x_2dx_3\wedge dx_1]_N$ & (\ref{s1})$\wedge\, dx_1$\\\hline

    6. &  $[2x_2dx_1 \wedge dx_2]_N=(\mu-3)[x_3^{\mu-4}dx_3\wedge dx_1]_N$ & (\ref{s2})$\wedge\, dx_1$ \\ \hline

    7. & $[x_1^2dx_2 \wedge dx_3]_N=0$ & \begin{tabular}{c} rows 1. and 2. \\
       and $[x_1^2]_N=[x_2^2+x_3^{\mu-3}]_N$ \end{tabular} \\ \hline

    8. & $[x_3^2dx_1 \wedge dx_2]_N=0$ & (\ref{s1})$\wedge\, x_3 dx_1$  and $[x_2 x_3]_N=0$\\  \hline

    9. & $[x_2^2dx_1 \wedge dx_2]_N=0$ & (\ref{s2})$\wedge\, x_2 dx_1$  and $[x_2 x_3]_N=0$\\  \hline

\end{tabular}
\end{center}
\smallskip
\caption{\small Relations towards calculating $[\Lambda^2(\mathbb R^{2n})]_N$ for $N=S_{\mu}$}\label{tabs1}
\end{table}
\end{small}

Table \ref{tabs1} and Proposition \ref{d-wedge} easily imply the
following proposition:

\begin{prop}
\label{s-all}
$[\Lambda ^{2}(\mathbb R^{2n})]_{S_{\mu}}$ is a $\mu +1$-dimensional vector space spanned
by the algebraic restrictions to $S_{\mu}$ of the $2$-forms

\smallskip

$\theta _1=dx_1\wedge dx_3,\;\;   \theta _2=dx_2\wedge dx_3, \;\; \theta_3 = dx_1\wedge dx_2,$

\smallskip

$\sigma _1 = x_3dx_1\wedge dx_2,\;\;  \sigma _2 = x_1 dx_2\wedge dx_3, $

\smallskip

$\theta _{4+k}= x_3^kdx_1\wedge dx_3$,\;\; dla $1\leq k\leq \mu-4$.

\end{prop}

Proposition \ref{s-all} and results of section \ref{method}  imply
the following description of the space $[Z^2(\mathbb R^{2n})]_{S_{\mu}}$ and the manifold $[{\rm Symp} (\mathbb R^{2n}]_{S_{\mu}}$.

\begin{prop}
\label{s_baza} The space $[Z^2(\mathbb R^{2n})]_{S_{\mu}}$ has dimension $\mu$.
It is spanned by the algebraic restrictions to $S_{\mu}$ of the $2$-forms
$$\theta _1,\theta_2,\theta _3,\; \ \theta _4 = \sigma _1- \sigma _2,\;\;\theta _{4+k}= x_3^kdx_1\wedge dx_3,\;\; for\;\; 1\leq k\leq \mu-4.$$
If $n\ge 3$ then $[{\rm Symp} (\mathbb R^{2n})]_{S_{\mu}} = [Z^2(\mathbb R^{2n})]_{S_{\mu}}$.
The manifold $[{\rm Symp} (\mathbb R^{4})]_{S_{\mu}}$ is an open part of the $\mu$-space $[Z^2(\mathbb R^{4})]_{S_{\mu}}$
consisting of algebraic restrictions of the form $[c_1\theta _1 + \cdots + c_{\mu}\theta _{\mu}]_{S_{\mu}}$ such that $(c_1,c_2,c_3)\ne (0,0,0)$.
\end{prop}


\begin{thm}
\label{klas_s} $ \ $

\smallskip

\noindent (i) \ Any algebraic restriction in $[Z^2(\mathbb R^{2n})]_{S_{\mu}}$ can be brought by a symmetry of $S_{\mu}$ to one of the normal forms $[S_{\mu}]^{i,j}$ given in the second column of Table \ref{tabs};

\smallskip

\noindent (ii) \ The singularity classes corresponding to the normal forms are disjoint;

\smallskip

\noindent (iii) \ The parameters $c_i$ of the normal forms $[S_{\mu}]^{i,j}$ are moduli.

\smallskip

\noindent (iv)  \ The codimension in $[Z ^2(\mathbb R^{2n})]_{S_{\mu}}$ of the singularity class
corresponding to the normal form  $[S_{\mu}]^{i,j}$ is equal to $i$.

\end{thm}

\renewcommand*{\arraystretch}{1.3}
\begin{center}
\begin{table}[!h]

    \begin{small}
    \noindent
    \begin{tabular}{|p{2.9cm}|p{5.7cm}|c|c|c|}
            \hline
    Symplectic class &   Normal forms for
                algebraic restrictions    & cod & $\mu ^{\rm sym}$ &  ind  \\ \hline
  $(S_{\mu})^{0}$ \;\;\;\; \; $(2n\ge 4)$ & $[S_{\mu}]^0: [\theta _1 + c_2\theta _2 + c_2\theta_3]_{S_{\mu}}$  &  $0$ & $2$ & $0$  \\  \hline

  $(S_{\mu})^{k}_2 \;\;\;\;\;\; (2n\ge 4)$ \newline for $1\leq k\leq \mu-\!5$ & $[S_{\mu}]^{k}_2: [\theta _2 + c_3\theta _3 + c_{4+k}\theta _{4+k}]_{S_{\mu}}$
  \newline 
  $c_{4+k}\!\ne\! 0$        &  $k$ & $k+2$ & $0$ \\ \hline

   $(S_{\mu})^{\mu-4}_2$ \; $(2n\ge 4)$& $[S_{\mu}]^{\mu-4}_2: [\theta _2 + c_3\theta_3+c_{\mu}\theta_{\mu}]_{S_{\mu}}$, \; $c_3c_{\mu}=0$
              & $\mu-4$ & $\mu-3$ & $0$     \\ \hline

   $(S_{\mu})^{1+k}_r$ \; $(2n\ge 4)$ \newline for  $1\leq k\leq \mu-\!6$ & $[S_{\mu}]^{1+k}_r: [\theta _3\!+c_{4+k}\theta _{4+k}\!+c_{5+k} \theta _{5+k}]_{S_{\mu}}$ \newline $c_{4+k}\!\ne\! 0$
                                         &  $k+1$ & $k+3$ & $0$ \\ \hline

    $(S_{\mu})^{\mu-4}_r$ \;\; $(2n\ge 4)$ & $[S_{\mu}]^{\mu-4}_r: [\theta _3 + c_{\mu-1}\theta _{\mu-1}]_{S_{\mu}}$ & $\mu-4$ & $\mu-3$ & $0$  \\ \hline

    $(S_{\mu})^{3,1}$ \;\;\; $(2n\ge 6)$ & $[S_{\mu}]^{3,1}: [c_4\theta _4 + \theta _5]_{S_{\mu}}$
     &  $3$ & $4 $ & $1$    \\ \hline

    $(S_{\mu})^{2+k,1}$ \; $(2n\ge 6)$ \newline for  $2\leq k\leq \mu -4$ & $[S_{\mu}]^{2+k,1}: [\theta _4 + c_{4+k}\theta _{4+k}]_{S_{\mu}}$ \newline $c_{4+k}\ne 0$ for  $2\leq k\leq \mu -5$
     &  $k+2$ & $k+3 $ & $1$    \\ \hline

    $(S_{\mu})^{3+k,k}$ \, $(2n\ge 6)$ \newline for  $2\leq k\leq \mu -4$ & $[S_{\mu}]^{3+k,k}: [\theta _{4+k}]_{S_{\mu}}$ \;\;for  $2\leq k\leq \mu -4$ &  $k+3$ & $k+3$ & $k$ \\ \hline

     $(S_{\mu})^{\mu}$ \;\;\;\;\; $(2n\ge 6)$ & $[S_{\mu}]^{\mu}: [0]_{S_{\mu}}$ &  $\mu$ & $\mu$ & $\infty $ \\ \hline
\end{tabular}

\smallskip

\caption{\small Classification of symplectic $S_{\mu}$ singularities.  \newline
$cod$ -- codimension of the classes; \ $\mu ^{sym}$-- symplectic multiplicity; \newline $ind$ --  index of isotropness.}\label{tabs}

\end{small}
\end{table}
\end{center}
\medskip

The proof of Theorem \ref{klas_s} is presented in section
\ref{s-proof}.

In the first column of Table \ref{tabs}  by $(S_{\mu})^{i,j}$ we denote a subclass of $(S_{\mu})$ consisting of $N\in (S_{\mu})$ such that the algebraic restriction $[\omega ]_N$ is diffeomorphic to some algebraic restriction of the normal form $[S_{\mu}]^{i,j}$ where $i$ is the codimension of the class and $j$ denotes index of isotropness of the class. Classes $(S_{\mu})^{i}_2$ and $(S_{\mu})^{i}_r$ can be distinguished geometrically (see section \ref{s-geom_cond}) and by relative Lagrangian tangency order - $L_{2:1}$ defined in  section \ref{s-lagr} (remark \ref{s-relative}). The classes $(S_{\mu})^{i}_2$ have $L_{2:1}=\frac{2}{\lambda_{\mu}}$ and the classes $(S_{\mu})^{i}_r$ have $L_{2:1}=\frac{r}{\lambda_{\mu}}$ where $\lambda_{\mu}\!=1$ for even $\mu$ and $\lambda_{\mu}\!=2$ for odd $\mu$.

Theorem \ref{thm A}, Theorem \ref{klas_s} and Proposition \ref{s_baza} imply the following statement.

\begin{prop}
\label{def-classes-s} The classes $(S_{\mu})^{i,j}$ are symplectic
singularity classes, i.e. they are closed with respect to the
action of the group of symplectomorphisms. The class $(S_{\mu})$ is
the disjoint union of the classes $(S_{\mu})^{i,j}$.
The classes $(S_{\mu})^0, (S_{\mu})^{i}_2$ and $(S_{\mu})^{i}_r$ for $1\leq i\leq \mu -4$ are non-empty for
any dimension $2n\ge 4$ of the symplectic space; the classes
$(S_{\mu})^{i,1}$ for $3\leq i\leq \mu -2$ and $(S_{\mu})^{i,i-3}$ for $5\leq i \leq \mu -1$ and $(S_{\mu})^{\mu}$ are empty if $n=2$ and not empty if $n\ge 3$.
\end{prop}

\subsection{Symplectic normal forms. Proof of Theorem \ref{s-main}}
\label{s-normal}

Let us transfer the normal forms $[S_{\mu}]^{i,j}$  to symplectic normal forms using Theorem \ref{th-all-closed},
i.e. realizing the algorithm in section \ref{method}. Fix a family $\omega ^{i,j}$ of symplectic
forms on $\mathbb R^{2n}$ realizing the family $[S_{\mu}]^{i,j}$ of algebraic restrictions.  We can fix, for example

\smallskip

\noindent $\omega ^0 = \theta _1 + c_2\theta _2 + c_3\theta _3 +
dx_2\wedge dx_4 + \sum_{i=3}^n dx_{2i-1} \wedge dx_{2i};$

\smallskip

\noindent $\omega ^{k}_2\! = \theta _2 + c_3\theta _3 + c_{4+k}\theta _{4+k} +
 dx_1\wedge dx_4+ \sum_{i=3}^n dx_{2i-1} \wedge dx_{2i}, \  c_{4+k}\!\ne\! 0, \ 1\leq k\leq \mu-\!5; $

 \smallskip

\noindent $\omega ^{\mu-4}_2\! = \theta _2 + c_3\theta _3 + c_{\mu}\theta_{\mu} +
 dx_1\wedge dx_4+ \sum_{i=3}^n dx_{2i-1} \wedge dx_{2i}, \ \ c_3 c_{\mu}=0;$

 \smallskip

\noindent $\omega ^{1+k}_r\!= \theta_3\!+ c_{4+k}\theta_{4+k}\!+ c_{5+k}\theta _{5+k}\!+
 \sum_{i=2}^n dx_{2i-1} \wedge dx_{2i}, \ c_{4+k}\!\ne\! 0, \ 1\!\leq\! k\!\leq\! \mu-\!6;$

 \smallskip

\noindent $\omega ^{\mu-4}_r\! =  \theta _3 + c_{\mu-1} \theta _{\mu-1} +
 dx_4\wedge dx_3+ \sum_{i=3}^n dx_{2i-1} \wedge dx_{2i}; $

\smallskip

\noindent $\omega ^{3,1}\! = c_4\theta _4\!+ \theta _{5}\!+ \sum_{i=1}^3 dx_i \wedge dx_{i+3}\!+ \sum_{i=4}^n dx_{2i-1} \wedge dx_{2i};$

\smallskip

\noindent $\omega ^{2+k,1}\! = \theta _4\!+ c_{4+k}\theta _{4+k}\!+ \sum_{i=1}^3 dx_i \wedge dx_{i+3}\!+ \sum_{i=4}^n dx_{2i-1} \wedge dx_{2i},  \ 2\!\leq\!k\!\leq\!\mu-\!4;$

\smallskip


\noindent $\omega ^{3+k,k}\! = \theta _{4+k} + \sum_{i=1}^3 dx_i \wedge dx_{i+3}+ \sum_{i=4}^n dx_{2i-1} \wedge dx_{2i}, \  2\leq k\leq \mu-\!4;$

\smallskip

\noindent $\omega ^{\mu} = \sum_{i=1}^3 dx_i \wedge dx_{i+3}+ \sum_{i=4}^n dx_{2i-1} \wedge dx_{2i}.$

\medskip

 Let $\omega = \sum_{i=1}^m dp_i \wedge dq_i$, where
$(p_1,q_1,\cdots,p_n,q_n)$ is the coordinate system on $\mathbb
R^{2n}, n\ge 3$ (resp. $n=2$). Fix a family $\Phi ^{i,j}$ of local diffeomorphisms which bring the family of symplectic forms $\omega ^{i,j}$ to the symplectic form $\omega $: $(\Phi ^{i,j})^*\omega ^{i,j} = \omega $. Consider the families $S_{\mu}^{i,j} = (\Phi ^{i,j})^{-1}(S_{\mu})$. Any stratified submanifold of the symplectic space $(\mathbb R^{2n},
\omega )$ which is diffeomorphic to $S_{\mu}$ is symplectically
equivalent to one and only one of the normal forms $S_{\mu}^{i,j}$ presented in Theorem
\ref{s-main}. By Theorem \ref{klas_s} we obtain that  parameters
$c_i$ of the normal forms are moduli.

\subsection{Distinguishing symplectic classes of $S_{\mu}$ by Lagrangian tangency orders}
\label{s-lagr}

Lagrangian tangency orders will be used to obtain a more detailed
classification of $(S_{\mu})$. A curve $N\in (S_{\mu})$ may be described
as a union of two invariant components $C_1$ and $C_2$. $C_1$ is diffeomorphic to $A_1$ singularity and consists of two parametrical branches $B_{1+}$ and $B_{1-}$. $C_2$ is diffeomorphic to $A_{\mu-4}$ singularity and consists of one parametric branch if $\mu$ is even number and consists of two branches $B_{2+}$ and $B_{2-}$ if $\mu$ is odd number. The parametrization of these branches is given in the second column of Table
\ref{tabseven-lagr} or Table
\ref{tabsodd-lagr}. To distinguish the classes of this singularity completely we need following three invariants:
\begin{itemize}
  \item $Lt(N)=Lt(C_1,C_2)
      $
  \item $L_1=Lt(C_1)=\max\limits _L (\min \{t(B_{1+},L), t(B_{1-},L)\})$
  \item $L_2=Lt(C_2)
  $
\end{itemize}

where $L$ is a smooth Lagrangian submanifold of the symplectic
space.

\medskip

 Considering the triples $(Lt(N), L_1, L_2)$ we obtain detailed classification of
symplectic singularities of $S_{\mu}$. Some subclasses appear (see Table
\ref{tabseven-lagr} and \ref{tabsodd-lagr}) having a natural geometric interpretation
(Table \ref{tabs-geom}).

\begin{thm}
\label{lagr-s} A stratified submanifold $N\in (S_{\mu})$ of a
symplectic space $(\mathbb R^{2n}, \omega )$ with the canonical
coordinates $(p_1, q_1, \cdots, p_n, q_n)$ is symplectically
equivalent to one and only one of the curves presented in the
second column of Table \ref{tabseven-lagr} or \ref{tabsodd-lagr}. The parameters $c_i$ are moduli. The Lagrangian tangency orders of the curve are presented in  the fifth, the sixth and the seventh column of these tables and the codimension of the classes is given in the fourth column.
\end{thm}

\setlength{\tabcolsep}{1mm}

\begin{center}
\begin{table}[h]

    \begin{small}
    \noindent
    \begin{tabular}{|p{1.3cm}|p{4.6cm}|p{2.3cm}|c|c|c|c|}
                    \hline
    Class &  Parametrization  of branches \newline $B_{1\pm}$ and $C_2$ & Conditions \newline for subclasses  & cod & $Lt(N)$   &  $L_1$ & $L_2$ \\ \hline

        $(S_{\mu})^0$ & $(t,0,\pm t,-c_3t,0,\cdots )$ & $c_3\ne 0$& $0$ & $1$ & $1$ & $r$ \\ \cline{3-7}
     $2n\ge 4$ & $(t^r\!,t^2\!,0,c_2t^2\!-c_3t^r\!,0,\cdots )$ & $c_3=0$ & $1$ & $2$ & $\infty$ & $r$ \\ \hline

$(S_{\mu})^{k}_2$  & $(t,0, \pm t, c_3 t,0,\cdots)$ &  $c_{4+k}\cdot c_3\!\ne\!0$ & $k$ & $1$ & $1$ & $r\!+\!2k$   \\ \cline{3-7}
$2n\ge 4$  & $(0,t^2,t^r, \frac{c_{4+k}}{k+1} t^{2+2k}, 0,\cdots)$ & $c_3\!=0, c_{4+k}\!\ne 0$  & $k\!+\!1$ & $2$ & $\infty$  & $r\!+\!2k$ \\ \hline

$(S_{\mu})^{\mu-4}_2$  & $(t,0, \pm t, c_3 t,0,\cdots)$ & $c_3\ne 0$   & $\mu\!-\!4$ & $1$ & $1$ & $\infty$   \\ \cline{3-7}
$2n\ge 4$  & $(0,t^2,t^r, \frac{c_{\mu}}{r} t^{2r}, 0,\cdots)$ & $c_3=0$ & $\mu\!-\!3$ & $2$  & $\infty$ & $\infty$ \\ \hline

    $(S_{\mu})^{1+\!k}_r$ & $(t, \pm t, 0, 0,\cdots)$ & $c_{k+4}\ne 0$ & $k\!+\!1$ & $1$ & $1$ & $r\!+\!2k$ \\
    $2n\ge 4$ & $(t^r\!,\! 0,\! (c_{4\!+\!k}\!+\!c_{5\!+\!k}t^2)t^{r+2k}\!,t^2\!, 0,\!\cdots)$ & $1\!\leq k \leq \mu\! -6$  &   &  &  &  \\ \hline

    $(S_{\mu})^{\mu-4}_r$  & $(t, \pm t,0, 0,\cdots)$ & $c_{\mu-1}\ne 0$  & $\mu\!-\!4$ & $1$ & $1$ & $3r\!-\!4$   \\ \cline{3-7}
 $2n\ge 4$  & $(t^r,0, c_{\mu-1}t^{3r-4}, t^2, 0,\cdots)$ & $c_{\mu-1}=0$ & $\mu\!-\!3$ &  $1$ & $1$ & $\infty$ \\   \hline

  $(S_{\mu})^{3,1}$  & $(t,0,\pm t,0, 0,\cdots)$ &   & $3$ & $r+2$ & $\infty$ & $r+2$   \\
 $2n\ge 6$  & $( t^r,\frac{1}{2}t^4,0,-c_4t^{r+2},t^2, 0,\cdots)$ &  &  &  &  &  \\ \hline

$(S_{\mu})^{2+\!k,1}$ & $(t,0,\pm t,0,0, 0, \cdots)$ & $c_{4+k}\ne 0$ & $k\!+\!2$ & $r+2$ & $\infty$ & $r\!+\!2k$ \\ 
    $2n\ge 6$ & $(\!t^r\!,\frac{c_{4+k}t^{2(k+1)}}{k+1},\!0,\!-t^{r\!+\!2}\!,t^2\!,\!0,\! \cdots)$ & $2\!\leq k \leq \mu\! -5$  &   &  &  &  \\ \cline{3-7}
      &     &   $k=\mu-4$ & $\mu\!-\!2$ &  $r+2$ & $\infty$ & $\infty$ \\   \hline

     $(S_{\mu})^{3+k,k}$ & $(t,0, \pm t,0,0, 0, \cdots)$ & $2\!\leq k \leq \mu\! -5$ & $\!k\!+3$ & $r\!+\!2k$ & $\infty$ & $r\!+\!2k$ \\ \cline{3-7}
    $2n\ge 6$ & $(t^{r},\frac{t^{2(k+1)}}{k+1}, 0, 0, t^2, 0, 0, \cdots)$ & $k=\mu-4$ & $\mu\!-\!1$ &  $3r-2$ & $\infty$ & $\infty$ \\   \hline

    $(S_{\mu})^{\mu}$  & $(t, 0,\pm t,0, 0,\cdots)$ &   & $\mu$ & $\infty$ & $\infty$ & $\infty$   \\ 
$2n\ge 6$  & $(t^r, 0, 0 , 0, t^2,0 ,\cdots)$ &  &  &  &  &  \\ \hline
\end{tabular}

\smallskip

\caption{\small Lagrangian tangency orders for symplectic classes of $S_{\mu}$ singularity ($\mu$ even).}\label{tabseven-lagr}

\end{small}
\end{table}
\end{center}

\medskip

\setlength{\tabcolsep}{1mm}

\begin{center}
\begin{table}[h]

    \begin{small}
    \noindent
    \begin{tabular}{|p{1.3cm}|p{4.7cm}|p{2.3cm}|c|c|c|c|}
                      \hline
    Class &  Parametrization  of branches \newline $B_{1\pm}$ and $B_{2\pm}$ & Conditions \newline for subclasses  & cod & $Lt(N)$   &  $L_1$ & $L_2$ \\ \hline

        $(S_{\mu})^0$ & $(t,0,\pm t,-c_3t,0,\cdots )$ & $c_3\ne 0$& $0$ & $1$ & $1$ & ${\frac{r}{2}}$ \\ \cline{3-7}
     $2n\ge 4$ & $(\pm t^{\frac{r}{2}}\!,t\!,0,c_2t^2\!\mp c_3t^{\frac{r}{2}}\!,0,\cdots )$ & $c_3=0$ & $1$ & $1$ & $\infty$ & ${\frac{r}{2}}$ \\ \hline

$(S_{\mu})^{k}_2$  & $(t,0, \pm t, c_3 t,0,\cdots)$ &  $c_{4+k}\cdot c_3\!\ne\!0$ & $k$ & $1$ & $1$ & $\frac{r}{2}\!+\!k$   \\ \cline{3-7}
$2n\ge 4$  & $(0,t,\pm t^{\frac{r}{2}}, \frac{c_{4+k}}{k+1} t^{1+k}, 0,\cdots)$ & $c_3\!=0, c_{4+k}\!\ne 0$  & $k\!+\!1$ & $1$ & $\infty$  & $\frac{r}{2}\!+\!k$ \\ \hline

$(S_{\mu})^{\mu-4}_2$  & $(t,0, \pm t, c_3 t,0,\cdots)$ & $c_3\ne 0$   & $\mu\!-\!4$ & $1$ & $1$ & $\infty$   \\ \cline{3-7}
$2n\ge 4$  & $(0,t,t^{\frac{r}{2}}, \frac{c_{\mu}}{r} t^{r}, 0,\cdots)$ & $c_3=0$ & $\mu\!-\!3$ & $1$  & $\infty$ & $\infty$ \\ \hline

    $(S_{\mu})^{1+\!k}_r$ & $(t, \pm t, 0, 0,\cdots)$ & $c_{k+4}\ne 0$ & $k\!+\!1$ & $1$ & $1$ & $\frac{r}{2}\!+\!k$ \\
    $2n\ge 4$ & $(\pm t^{\frac{r}{2}}\!,\! 0,\!\pm (c_{4\!+\!k}\!+\!c_{5\!+\!k}t)t^{\frac{r}{2}+k}\!,t\!, 0,\!\cdots)$ & $1\!\leq k \leq \mu\! -6$  &   &  &  &  \\ \hline

    $(S_{\mu})^{\mu-4}_r$  & $(t, \pm t,0, 0,\cdots)$ & $c_{\mu-1}\ne 0$  & $\mu\!-\!4$ & $1$ & $1$ & $\frac{3r}{2}\!-\!2$   \\ \cline{3-7}
 $2n\ge 4$  & $(\pm t^{\frac{r}{2}},0,\pm c_{\mu-1}t^{\frac{3r-4}{2}},t, 0,\cdots)$ &  $c_{\mu-1}=0$ & $\mu\!-\!3$ &  $1$ & $1$ & $\infty$ \\   \hline

  $(S_{\mu})^{3,1}$  & $(t,0,\pm t,0, 0,\cdots)$ &   & $3$ & $\frac{r}{2}\!+\!1$ & $\infty$ & $\frac{r}{2}\!+\!1$   \\
 $2n\ge 6$  & $(\pm t^{\frac{r}{2}},\frac{1}{2}t^2,0,\mp c_4t^{\frac{r+2}{2}},t, 0,\cdots)$ &  &  &  &  &  \\ \hline

$(S_{\mu})^{2+\!k,1}$ & $(t,0,\pm t,0,0, 0, \cdots)$ & $c_{4+k}\ne 0$ & $k\!+\!2$ & $\frac{r}{2}\!+\!1$ & $\infty$ & $\frac{r}{2}\!+\!k$ \\ 
    $2n\ge 6$ & $(\pm\!t^\frac{r}{2}\!,\frac{c_{4+k}t^{k+1}}{k+1},\!0,\!\mp t^{\frac{r+2}{2}}\!,t\!,\!0,\! \cdots)$ & $2\!\leq k \leq \mu\! -5$  &   &  &  &  \\ \cline{3-7}
      &     &   $k=\mu-4$ & $\mu\!-\!2$ &  $\frac{r}{2}\!+\!1$ & $\infty$ & $\infty$ \\   \hline

     $(S_{\mu})^{3+k,k}$ & $(t,0, \pm t,0,0, 0, \cdots)$ & $2\!\leq k \leq \mu\! -5$ & $\!k\!+3$ & $\frac{r}{2}\!+\!k$ & $\infty$ & $\frac{r}{2}\!+\!k$ \\ \cline{3-7}
    $2n\ge 6$ & $(\pm t^{\frac{r}{2}},\frac{t^{k+1}}{k+1}, 0, 0, t, 0, 0, \cdots)$ & $k=\mu-4$ & $\mu\!-\!1$ &  $\frac{3}{2}r\!-\!1$ & $\infty$ & $\infty$ \\   \hline

    $(S_{\mu})^{\mu}$  & $(t, 0,\pm t,0, 0,\cdots)$ &   & $\mu$ & $\infty$ & $\infty$ & $\infty$   \\ 
$2n\ge 6$  & $(\pm t^{\frac{r}{2}}, 0, 0 , 0, t,0 ,\cdots)$ &  &  &  &  &  \\ \hline
\end{tabular}

\smallskip

\caption{\small Lagrangian tangency orders for symplectic classes of $S_{\mu}$ singularity ($\mu$ odd).}\label{tabsodd-lagr}

\end{small}
\end{table}
\end{center}
\medskip

\begin{rem}
The numbers $L_1$ and $L_2$ can be easily calculated knowing Lagrangian tangency orders for $A_1$ and $A_{\mu-4}$ singularities (see Table 2 in \cite{DT}) or by direct applying the definition of the Lagrangian tangency order and finding the nearest Lagrangian submanifold to components. Next we calculate $Lt(N)$ by definition knowing that it can not be greater than $\min(L_1,L_2)$.

We can compute $L_1$ using the algebraic restrictions $[\omega ^{i,j}]_{C_1}$ where the space $[Z^2 (\mathbb R^{2n})]_{C_1}$ is spanned only by the algebraic restriction to $C_1$ of the $2$-form $\theta_3$. For example for the class $(S_{\mu})^0$ we have $[\theta_1+c_2\theta_2+c_3\theta_3]_{C_1}=[c_3\theta_3]_{C_1}$ and thus $L_1=1$ when $c_3\ne0$ and $L_1=\infty$ when $c_3=0$.

We can compute $L_2$ using the algebraic restrictions $[\omega ^{i,j}]_{C_2}$ where the space $[Z^2 (\mathbb R^{2n})]_{C_2}$ is spanned only by the algebraic restrictions to $C_2$ of the $2$-forms $\theta_1, \,\theta_{4+k}$ for $k=1,2,\ldots,\theta_{\mu-1}$. For example for the class $(S_{\mu})^0$ we have $[ \theta_1+c_2\theta_2+c_3\theta_3]_{C_2}=[\theta_1]_{C_2}$ and thus $L_2= \mu-3$ if $\mu$ is an even number and $L_2= \frac{\mu-3}{2}$ if $\mu$ is an odd number.

$Lt(N)\leq1=\min(L_1,L_2)$ when $c_3\ne0$. Applying the definition of $Lt(N)$ we find the smooth Lagrangian submanifold $L$ described by the conditions: $\;p_i=0,\; i\in \{1,\ldots,n\}$ and we get $Lt(N)\geq t(N,L)=1$ in this case.

If $c_3=0$ then $Lt(N)\leq L_2=\min(L_1,L_2)$, but applying the definition of $Lt(N)$ we have $t(N,L)\leq2$ (resp. $t(N,L)\leq1$) for all Lagrangian submanifolds $L$. For $L$ described by the conditions: $\;q_i=0,\; i\in \{1,\ldots,n\}$ we get $Lt(N)= t(N,L)=2$ if $\mu$ is an even number and $Lt(N)= t(N,L)=1$ if $\mu$ is an odd number.

\end{rem}

\begin{rem} \label{s-relative}
We are not able to distinguish some classes $(S_{\mu})^{i}_2$ and $(S_{\mu})^{i}_r$ by the triples $(Lt(N), L_1, L_2)$ but we can do it using relative Lagrangian tangency orders.

\smallskip

\noindent We define $L_{2:1}=Lt[C_2:B_{1\pm}]=\max(Lt[C_2:B_{1+}],Lt[C_2:B_{1-}])$.

\smallskip

\noindent Since branches $B_{1+}$ and $B_{1-}$ are smooth curves then $Lt(B_{1+})=Lt(B_{1-})=\infty$ and $L_{2:1}=\max\limits _L (t(C_{2},L))$ where $L$ is a smooth Lagrangian submanifold containing $B_{1+}$ or $B_{1-}$.

\noindent Considering such smooth Lagrangian submanifolds  we obtain
$L_{2:1}=\frac{2}{\lambda_{\mu}}$ for the classes $(S_{\mu})^{i}_2$ and $L_{2:1}=\frac{\mu-3}{\lambda_{\mu}}$ for the classes $(S_{\mu})^{i}_r$ ($\lambda_{\mu}\!=1$ for even $\mu$ and $\lambda_{\mu}\!=2$ for odd $\mu$).

\end{rem}

\subsection{Geometric conditions for the classes $(S_{\mu})^{i,j}$}
\label{s-geom_cond}

The classes $(S_{\mu})^{i,j}$ can be distinguished geometrically, without using any local coordinate system.

Let $N\in (S_{\mu})$. Then $N$ is the union of two singular $1$-dimensional irreducible components diffeomorphic to $A_1$ and $A_{\mu-4}$ singularities. In local coordinates they have the form
\[\mathcal{C}_1=\{ x_1^2-x_2^2=0,\; x_{\geq 3}=0\},\]
\[\mathcal{C}_2=\{ x_1^2-x_3^{\mu-3}=0,\; x_2=x_{\geq 4}=0\}.\]
 Denote by $\ell _{1+}, \ell _{1-}$ the tangent lines at $0$ to the branches $\mathcal {B}_{1+}$ and $\mathcal {B}_{1-}$ respectively. These lines span a $2$-space $P_1$. Denote by $\ell_2$ the tangent line at $0$ to the component $\mathcal {C}_{2}$ and let $P_2$ be $2$-space tangent at $0$ to component $\mathcal {C}_2$. Define  line $\ell_3 = P_1\cap P_2$.
 The lines $\ell _{1\pm}, \ell _2$ span a  $3$-space $W=W(N)$. Equivalently $W$ is  the tangent space at $0$ to some (and then any) non-singular $3$-manifold containing $N$.
The classes $(S_{\mu})^{i,j}$ satisfy special conditions in terms of the restriction $\omega\vert_ W $, where $\omega $ is the symplectic form.
For $N=S_{\mu}=$(\ref{defs}) it is easy to calculate
\begin{equation}
\label{lines} \ell _{1\pm} = \Span (\partial /\partial x_1 \pm \partial /\partial x_2),
\ \ell _2 = \Span (\partial /\partial x_3), \ \ell _3 = \Span (\partial
/\partial x_1).\end{equation}

\begin{thm}
\label{geom-cond-s} A stratified submanifold $N\in (S_{\mu})$ of a
symplectic space $(\mathbb R^{2n}, \omega )$ belongs to the class
$(S_{\mu})^{i,j}$ if and only if the couple $(N, \omega )$ satisfies corresponding
conditions in the last column of Table \ref{tabs-geom}.
\end{thm}

\begin{center}
\begin{table}[h]
    \begin{small}
    \noindent
    \begin{tabular}{|p{1.2cm}|p{4.6cm}|p{6cm}|}
            \hline
    Class &  Normal form & Geometric conditions  \\ \hline

   $(S_{\mu})^0$ & & $ \omega|_{\ell_2+\ell_3} \ne 0$  \\ \cline{2-3}
         & $[S_{\mu}]^0_0: [\theta _1 + c_2\theta _2 + c_3\theta _3]_{S_{\mu}}$
   \newline $ c_3\ne 0$ & $\omega|_{\ell_{1+}+\ell_{1-}}\! \ne 0$ \; and  none of components is contained in a Lagrangian submanifold \\ \cline{2-3}
     & $[S_{\mu}]^0_1: [\theta _1 +c_2\theta _2 ]_{S_{\mu}}$ & $\omega |_{\ell_{1+}+\ell_{1-}}\!=0$ \; (so component $C_1$ is contained in a Lagrangian submanifold) \\ \hline \hline

  $(S_{\mu}\!)^{i}_2$ & &  $\omega|_{\ell_2+\ell_3} = 0$ but $\omega|_{\ell_{1\pm}+\ell_2} \ne 0$ \\ \hline
 $(S_{\mu}\!)^{k}_2$  & $[S_{\mu}]^{k}_2\!: [\theta _2\!+\! c_3\theta_3\!+\!c_{4+k}\theta_{4+k}]_{S_{\mu}}$
                                                    \newline $c_3\cdot c_{4+k} \ne 0$ for $1\leq k\leq\mu-5$ &  \ $\omega|_{\ell_{1+}+\ell_{1-}} \ne 0$ and $L_2=\frac{r+2k}{\lambda_{\mu}}$  \\ \cline{2-3}
 & $[S_{\mu}]^{k}_{2,L_1=\infty}\!: [\theta _2\!+\!c_{4+k}\theta_{4+k}]_{S_{\mu}}$
                                                    \newline $c_{4+k} \ne 0$ for $1\leq k\leq\mu-5$ & $\omega|_{\ell_{1+}+\ell_{1-}} = 0$ (so component $C_1$ is contained in a Lagrangian submanifold) and \ $L_2=\frac{r+2k}{\lambda_{\mu}}$  \\ \hline
 $(S_{\mu}\!)^{\mu-4}_2$ & $[S_{\mu}]^{\mu-4}_2: [\theta _2 +c_3\theta _3 ]_{S_{\mu}}$, $c_3\ne 0$ & $\omega|_{\ell_{1+}+\ell_{1-}} \ne 0$ and component $C_2$ is contained in a Lagrangian submanifold \\ \cline{2-3}

   & $[S_{\mu}]^{\mu-3}_2: [\theta _2  +c_{\mu} \theta _{\mu}]_{S_{\mu}}$ & $\omega|_{\ell_{1+}+\ell_{1-}} = 0$, both components are contained in  Lagrangian submanifolds\\ \hline\hline

$(S_{\mu}\!)^{i}_r$ & &  $\omega|_{\ell_2+\ell_3} = 0$ and $\omega|_{\ell_{1\pm}+\ell_2} = 0$ \newline but $\omega|_{\ell_{1+}+\ell_{1-}} \ne 0$ \\ \hline

$(S_{\mu}\!)^{1+k}_r$ & $[S_{\mu}]^{1+k}_r\!:$ \newline $[\theta _3\!+\!c_{4\!+k}\theta_{4\!+k}\!+\! c_{5\!+k}\theta_{5\!+k}]_{S_{\mu}}$
                                                    \newline $c_{4+k} \ne 0$ for $1\leq k\leq \mu-6$ &  none of components is contained in a Lagrangian submanifold and $L_2=\frac{r+2k}{\lambda_{\mu}}$     \\ \hline

$(S_{\mu}\!)^{\mu-4}_r$ & $[S_{\mu}]^{\mu-4}_r: [\theta _3 +c_{\mu-1} \theta _{\mu-1} ]_{S_{\mu}}$ \newline $c_{\mu-1} \ne 0$ & none of components is contained in a Lagrangian submanifold and $L_2=\frac{3r-4}{\lambda_{\mu}}$  \\ \cline{2-3}

& $[S_{\mu}]^{\mu-3}_r: [\theta _3 ]_{S_{\mu}}$ & component $C_2$ is contained in  Lagrangian submanifolds \\ \hline \hline

& & $\omega\vert_ W = 0$ and  component $C_1$ is contained in a Lagrangian submanifold \\ \hline

$(S_{\mu}\!)^{3,1}$ & $[S_{\mu}]^{3,1}: [c_4\theta _4 + \theta_{5}]_{S_{\mu}}$
                                                     &   $L_2=\frac{r+2}{\lambda_{\mu}}$  and $Lt(N)=\frac{r+2}{\lambda_{\mu}}$   \\ \hline
$(S_{\mu}\!)^{2+k,1}$ & $[S_{\mu}]^{2+k,1}: [\theta _4 + c_{4+k}\theta_{4+k}]_{S_{\mu}}$
                                                    \newline $c_{4+k} \ne 0$ and $2\leq k\leq \mu-5$ &   $L_2=\frac{r+2k}{\lambda_{\mu}}$  and $Lt(N)=\frac{r+2}{\lambda_{\mu}}$   \\ \cline{2-3}

& $[S_{\mu}]^{\mu-2,1}: [\theta _4 + c \theta _{\mu} ]_{S_{\mu}}$ \  & both components are contained in  Lagrangian submanifolds  and $Lt(N)=\frac{r+2}{\lambda_{\mu}}$\\ \hline

$(S_{\mu}\!)^{3+k,k}$ & $[S_{\mu}]^{3+k,k}: [\theta _{4+k} ]_{S_{\mu}}$ \newline $2\leq k\leq \mu-5$ & $L_2=\frac{r+2k}{\lambda_{\mu}}$  and $Lt(N)=\frac{r+2k}{\lambda_{\mu}}$   \\ \cline{2-3}

& $[S_{\mu}]^{\mu-1}: [ \theta _{\mu} ]_{S_{\mu}}$ & both components are contained in  Lagrangian submanifolds  and $Lt(N)=\frac{3r-2}{\lambda_{\mu}}$\\ \hline

$(S_{\mu}\!)^{\mu}$& $[S_{\mu}]^{\mu}: [ 0 ]_{S_{\mu}}$ & both components are contained in the same Lagrangian submanifold\\ \hline

    \end{tabular}

\smallskip

\caption{\small Geometric interpretation of singularity classes of $S_{\mu}$;
$W$ - the tangent space to a non-singular $3$-dimensional manifold  in $(\mathbb{R}^{2n\ge 4},\omega)$ containing $N\in (S_{\mu})$, $\lambda_{\mu}=1$ for even $\mu$ and $\lambda_{\mu}=2$ for odd $\mu$.}\label{tabs-geom}

\end{small}
\end{table}
\end{center}

\begin{proof}[Proof of Theorem \ref{geom-cond-s}]
The conditions on the pair $(\omega, N)$ in the last column of Table \ref{tabs-geom} are disjoint. It suffices to prove that these conditions  the row of $(S_{\mu})^{i,j}$, are satisfied for any $N\in (S_{\mu})^{i,j}$. This  is a corollary of the following claims:

\smallskip

\noindent 1. Each of the conditions in the last column of Table \ref{tabs-geom} is invariant with respect to the action of the group of diffeomorphisms in the space of pairs $(\omega , N)$;

\smallskip

\noindent 2. Each of these conditions depends only on the algebraic restriction $[\omega ]_N$;

\smallskip

\noindent 3. Take the simplest $2$-forms $\omega ^{i,j}$ representing the normal forms $[S_{\mu}]^{i,j}$ for algebraic restrictions.  The pair $(\omega = \omega ^{i,j}, S_{\mu})$ satisfies the condition in the last column of Table \ref{tabs-geom}, the row of $(S_{\mu})^{i,j}$.

\medskip

 The first statement is obvious,  the second one follows from  Lemma \ref{zero-at-zero}.

To prove the third statement we note that in the case $N = S_{\mu} = (\ref{defs})$ one has $W = \Span (\partial/\partial x_1, \partial / \partial x_2, \partial / \partial x_3)$ and  $\ell _{1\pm} = \Span (\partial /\partial x_1 \pm \partial /\partial x_2)$, $\ell _2 = \Span (\partial /\partial x_3)$, $\ell _3 = \Span (\partial /\partial x_1)$. By simply calculation and observation of Lagrangian tangency orders we obtain that the conditions in  the last column of Table \ref{tabs-geom}, the row of $(S_{\mu})^{i,j}$ are satisfied.

\end{proof}


\subsection{Proof of the theorem \ref{klas_s}}
\label{s-proof}

\begin{proof}

In our proof we use vector fields tangent to $N\in S_{\mu}$. Any vector fields tangent to $N\in S_{\mu}$ may be described as $V=g_1E+g_2\mathcal{H}$ where $E$ is the Euler vector field and $\mathcal{H}$ is the Hamiltonian vector field and $g_1,g_2$ are functions. It was shown in \cite{DT} (Prop.6.13) that the action of the Hamiltonian vector field on any 1-dimensional complete intersection is trivial.

The germ of a vector field tangent to $S_{\mu}$ of non trivial action on algebraic restriction of closed 2-forms to  $S_{\mu}$ may be described as a linear combination germs of the following vector fields:
$X_0\!=E,\, X_1\!=x_1E,\, X_2\!=x_2E,\, X_3\!=x_3E,\,  X_{l+2}\!=x_3^lE$ for $1\!<\!l\!<\!\mu\!-\!3$, where $E$ is the Euler vector field
$E=\sum_{i=1}^3 \lambda _i x_i \partial /\partial x_i$ and $\lambda _i$ are weights for $x_i$.

\begin{prop} \label{s-infinitesimal}

When $\mu$ is an even number then the infinitesimal action of germs of quasi-homogeneous vector
fields tangent to $N$ on the basis of the vector space of
algebraic restrictions of closed $2$-forms to $N$ is presented in
Table \ref{infini s}.
\begin{footnotesize}
\begin{table}[h]
\begin{center}
\begin{tabular}{|p{1.7cm}|c|c|c|c|c|}

 \hline

  $\mathcal L_{X_i} [\theta_j]$ & $[\theta_1]$   &   $[\theta_2]$ &   $[\theta_3]$ & $[\theta_4]$ & $[\theta_{4+k}]$ for $0\!<\!k\!<\!r$ \\ \hline 

  $X_0\!=\!E$  & $(r\!+\!2)[\theta_1]$& $(r\!+\!2)[\theta_2]$ & $2r[\theta_3]$ & $(2r\!+\!2)[\theta_4]$ & $(r\!+\!2(k\!+\!1)) [\theta_{4+k}]$ \\ \hline

  $X_1\!=\!x_1E$  & $ [0]$ & $-(r\!+\!2)[\theta_4]$ & $ [0]$ & $ [0]$ & $ [0]$ \\  \hline

   $X_2\!=\!x_2E$ & $-r[\theta_4]$ &  $[0]$  &$\frac{-3r^2}{2}[\theta_{\mu}]$ & $[0]$ & $ [0]$ \\   \hline

   $X_3\!=\!x_3E$  & $(r\!+\!4)[\theta_5]$& $ [0]$ & $r[\theta_4]$  & $[0]$ &  $(r\!+\!2(k\!+\!2))[\theta_{5+k}]$ \\       \hline

 $X_{l+2}\!=\!x_3^lE$ \newline \;\; $l\!<r\!-\!k$
  &  $(r\!+\!2l\!+\!2)[\theta_{4\!+\!l}]$ & $ [0]$ & $[0]$ & $[0]$ &  $(r\!+\!2(k\!+\!l\!+\!1))[\theta_{4\!+\!k\!+\!l}]$ \\       \hline

$X_{l+2}\!=\!x_3^lE$ \newline \;\; $r\!-\!k\!\leq\!l\!\leq\!r\!-\!1$
  &  $(r\!+\!2l\!+\!2)[\theta_{4\!+\!l}]$ &$ [0]$ & $[0]$ & $[0]$ &  $[0]$ \\       \hline

\end{tabular}
\end{center}

\caption{Infinitesimal actions on algebraic restrictions of closed
2-forms to  $S_{\mu}$. $E=(\mu-3)x_1 \partial /\partial x_1+ (\mu-3)x_2 \partial /\partial x_2+ 2x_3 \partial /\partial x_3$}\label{infini s}
\end{table}
\end{footnotesize}

\end{prop}

\begin{rem} When $\mu$ is an odd number we obtain a very similar table , we only have  to divide by $2$ all coefficients in  Table \ref{infini s}. The next part of the proof is written for even $\mu$. In the case of odd $\mu$ we repeat the same scheme.
\end{rem}

\medskip

Let $\mathcal{A}=[ \sum_{l=1}^{\mu}c_{l} \theta_{l}]_{S_{\mu}}$
be the algebraic restriction of a symplectic form $\omega$.

\medskip

The first statement of Theorem \ref{klas_s} follows from the following lemmas.

\begin{lem}
\label{slem1} If \;$ c_1\ne 0$\; then the algebraic restriction  $\mathcal{A}=[ \sum_{l=1}^{\mu}c_{l} \theta_{l}]_{S_{\mu}}$ can be reduced by a symmetry of $S_{\mu}$ to an algebraic restriction $[\theta_1+ \widetilde{c}_2\theta_2+\widetilde{c}_3 \theta_3]_{S_{\mu}}$.
\end{lem}

\begin{proof}[Proof of Lemma \ref{slem1}]

We use the homotopy method to prove that  $\mathcal{A}$ is diffeomorphic to $[\theta_1+ \widetilde{c}_2\theta_2+\widetilde{c}_3 \theta_3]_{S_{\mu}}$.
Let $\mathcal{B}_t=[c_1 \theta_1+c_2 \theta_2+c_3 \theta_3+(1-t)\sum_{l=4}^{\mu}c_l \theta_l]_{S_{\mu}}$
\; for $t \in[0;1]$. Then $\mathcal{B}_0=\mathcal{A}$\; and \;$\mathcal{B}_1=[c_1 \theta_1+c_2 \theta_2+c_3 \theta_3]_{S_{\mu}}$.
 We prove that there exists a family $\Phi_t \in Symm(S_{\mu}),\;t\in [0;1]$ such that
 \begin{equation}
\label{proofslem11}   \Phi_t^*\mathcal{B}_t=\mathcal{B}_0,\;\Phi_0=id.
\end{equation}
Let $V_t$ be a vector field defined by $\frac{d \Phi_t}{dt}=V_t(\Phi_t)$. Then differentiating (\ref{proofslem11}) we obtain
  \begin{equation}
\label{proofslem12}   \mathcal L_{V_t} \mathcal{B}_t=[\sum_{l=4}^{\mu}c_l \theta_l].
\end{equation}
We are looking for $V_t$ in the form $V_t=\sum_{k=1}^{\mu-2}b_k(t) X_k$   where $b_k(t)$ for $k=1,\ldots,\mu-2$ are smooth functions $b_k:[0;1]\rightarrow \mathbb{R}$. Then by Proposition  \ref{s-infinitesimal} equation
(\ref{proofslem12}) has a form

\setlength{\arraycolsep}{0.1mm}
\begin{small}
\begin{equation}  \label{proofslem13}
\left[ \begin{array}{ccccccc}
\!-(r\!+\!2)c_2 & -rc_1 & rc_3 & 0 & 0 & 0 & 0 \\
0 & 0 & (r+4)c_1 & 0 & 0 & 0 & 0 \\
0 & 0 & (1\!-\!t)(r\!+\!6)c_5 & (r\!+\!6)c_1  & 0 & 0 & 0 \\
\vdots & \vdots & \vdots &  \ddots & 0 & 0 & 0 \\
0 & 0 & (1\!-\!t)(r\!+\!2k)c_{k\!+\!2} & \cdots & (r\!+\!2k)c_1 & 0 & 0 \\
0 & 0 & \vdots & \cdots &\vdots &\ddots & 0 \\
0 & -\frac{3r^2}{2}c_3 & 3(1\!-\!t)rc_{\mu-1} & \cdots & 3(1\!-\!t)rc_{\mu\!-\!k\!+\!1} & \cdots & 3rc_1
\end{array} \right]
\left[ \begin{array}{c} b_1(t) \\ b_2(t) \\ b_3(t) \\\vdots \\ b_{k+1}(t)\\ \vdots \\ b_{\mu-2}(t)  \end{array} \right] =
\left[ \begin{array}{c} c_4 \\ c_5 \\ c_6 \\ \vdots \\ c_{k+3} \\ \vdots \\  c_{\mu}  \end{array}  \right]
\end{equation}
\end{small}

\noindent If \;$c_1\ne 0$ we can solve (\ref{proofslem13}).

\noindent We obtain $b_3(t)=\frac{c_5}{c_1(r+4)}$ and we may choose any $b_1$.

\noindent Other functions $b_k$ are determined by that choice.

 \noindent Let $b_1(t)=0$. This imply $b_2(t)=\frac{rc_3b_3(t)-c_4}{rc_1}=\frac{c_3c_5}{(r+4)c_1^2}-\frac{c_4}{rc_1}$.

  \noindent Next $b_4(t)=\frac{c_6}{(r+6)c_1}-\frac{(1-t)}{c_1}c_5b_3(t)$,  \; $b_5(t)=\frac{c_7}{(r+8)c_1}-\frac{(1-t)}{c_1}(c_6b_3(t)+c_5b_4(t))$,

  \noindent  consequently $b_{k+1}(t)=\frac{c_{k+3}}{(r+2k)c_1}-\frac{(1-t)}{c_1}\sum_{l=3}^k c_{k+5-l}b_l(t)$ for $k<\mu-3$,

  \noindent and eventually $b_{\mu-2}(t)=\frac{c_{\mu}}{3rc_1}+ \frac{r}{2c_1}c_3b_2(t)-\frac{(1-t)}{c_1}\sum_{l=3}^{\mu-3} c_{\mu+2-l}b_l(t)$.

\smallskip

  Diffeomorphisms $\Phi_t$ may be obtained as a flow of vector field $V_t$.
The family $\Phi_t$ preserves $S_{\mu}$, because $V_t$ is tangent to $S_{\mu}$ and $\Phi_t^*\mathcal{B}_t=\mathcal{A}$.
Using the homotopy arguments we have $\mathcal{A}$ diffeomorphic to $ \mathcal{B}_1=[c_1 \theta_1+c_2 \theta_2+c_3 \theta_3]_{S_{\mu}}$.
By the condition $c_1\ne 0$ we have a diffeomorphism $\Psi \in Symm(S_{\mu})$ of the form
  \begin{equation}
\label{proofslem14}
\Psi:\,(x_1,x_2,x_3)\mapsto (c_1^{-\frac{r}{r+2}} x_1,c_1^{-\frac{r}{r+2}} x_2,c_1^{-\frac{2}{r+2}} x_3)
\end{equation}
and we obtain
\[ \Psi^*(\mathcal{B}_1)=[ \theta_1+ \frac{c_2}{c_1}\theta_2+c_3 c_1^{-\frac{2r}{r+2}} \theta_3]_{S_{\mu}} =
 [ \theta_1+ \widetilde{c}_2 \theta_2+\widetilde{c}_3 \theta_3]_{S_{\mu}}.\]
\end{proof}

\begin{lem}
\label{slem2} If \;$c_1= 0$\;and $c_2 \neq 0$ and $c_{4+k} \ne 0$ and $c_l=0$ for $5\leq l< 4+k$  then the algebraic restriction $\mathcal{A}=[ \sum_{l=1}^{\mu}c_{l} \theta_{l}]_{S_{\mu}}$ can be reduced by a symmetry of $S_{\mu}$ to an algebraic restriction $[\theta_2+ \widetilde{c}_3\theta_3+ +\widetilde{c}_{4+k} \theta_{4+k}]_{S_{\mu}}$.
\end{lem}

\begin{proof}[Proof of Lemma \ref{slem2}]
We use similar methods as above to prove that lemma. \par \noindent If $c_1=0$ and $c_2 \neq 0$ and $c_{4+k} \ne 0$ and $c_l=0$ for $5\leq l< 4+k$ then
$\mathcal{A}= [c_2 \theta_2+c_3 \theta_3+c_4 \theta_4+\sum_{l=4+k}^{\mu}c_{l} \theta_{l}]_{S_{\mu}}$. \par \noindent  Let $\mathcal{B}_t=[c_2 \theta_2+c_3 \theta_3+(1-t)c_4 \theta_4+c_{4+k} \theta_{4+k}+(1-t)\sum_{l=5+k}^{\mu}c_{l} \theta_{l}]_{S_{\mu}}$ \; for $t \in[0;1]$. Then $\mathcal{B}_0=\mathcal{A}$\; and
\;$\mathcal{B}_1=[c_2 \theta_2+c_3 \theta_3+c_{4+k} \theta_{4+k}]_{S_{\mu}}$.
 We prove that there exists a family $\Phi_t \in Symm(S_{\mu}),\;t\in [0;1]$ such that
 \begin{equation}
\label{proofslem21}   \Phi_t^*\mathcal{B}_t=\mathcal{B}_0,\;\Phi_0=id.
\end{equation}
Let $V_t$ be a vector field defined by $\frac{d \Phi_t}{dt}=V_t(\Phi_t)$. Then differentiating (\ref{proofslem21}) we obtain
  \begin{equation}
\label{proofslem22}   \mathcal L_{V_t} \mathcal{B}_t=[c_4 \theta_4+\sum_{l=5+k}^{\mu}c_{l} \theta_{l}]_{S_{\mu}}.
\end{equation}
We are looking for $V_t$ in the form $V_t=\sum_{k=1}^{\mu-2}b_k(t) X_k$   where $b_k(t)$ are smooth functions $b_k:[0;1]\rightarrow \mathbb{R}$ for $k=1,\ldots,\mu-2$.  Then by Proposition  \ref{s-infinitesimal} equation (\ref{proofslem22})
has a form

\setlength{\arraycolsep}{0.3mm}
\begin{small}
\begin{equation}  \label{proofslem23}
\left[ \begin{array}{ccccccc}
\!-(r\!+\!2)c_2 & 0 & rc_3 & 0\  & \cdots & \cdots & 0 \\
0 & 0 & (r\!+\!2k+4)c_{k\!+\!4} & 0 \  & \cdots & \cdots & 0 \\
0 & 0 & \vdots & \ddots & 0 &\cdots & 0 \\
0 & -\frac{3r^2}{2}c_3 & 3(1-t)rc_{\mu-1} & \cdots  & 3rc_{k\!+\!4} & 0 \cdots & 0
\end{array} \right]
\left[ \begin{array}{c} b_1(t) \\ b_2(t) \\ b_3(t) \\ \vdots  \\ b_{\mu-2}(t)  \end{array} \right] =
\left[ \begin{array}{c} c_4 
            \\ c_{k+5} \\ \vdots \\  c_{\mu}  \end{array}  \right]
\end{equation}
\end{small}

\noindent If \;$ c_2\ne 0$ we can solve (\ref{proofslem23}).
  Diffeomorphisms $\Phi_t$ may be obtained as a flow of vector field $V_t$.
The family $\Phi_t$ preserves $S_{\mu}$, because $V_t$ is tangent to $S_{\mu}$ and $\Phi_t^*\mathcal{B}_t=\mathcal{A}$.
Using the homotopy arguments we have that $\mathcal{A}$ is diffeomorphic to $ \mathcal{B}_1=[c_2 \theta_2+c_3 \theta_3+c_{4+k} \theta_{4+k}]_{S_{\mu}}$.
By the condition $c_2\ne 0$ we have a diffeomorphism $\Psi \in Symm(S_{\mu})$ of the form
  \begin{equation}
\label{proofslem24}
\Psi:\,(x_1,x_2,x_3)\mapsto (c_2^{-\frac{r}{r+2}} x_1,c_2^{-\frac{r}{r+2}} x_2,c_2^{-\frac{2}{r+2}} x_3)
\end{equation}
and we obtain
\[ \Psi^*(\mathcal{B}_1)=[\theta_2+ c_3 c_2^{-\frac{2r}{r+2}}\theta_3+c_{4+k} c_2^{-(1+\frac{2k}{r+2})} \theta_{4+k}]_{S_{\mu}} =
 [ \theta_2+ \widetilde{c}_3\theta_3+\widetilde{c}_{4+k} \theta_{4+k}]_{S_{\mu}}.\]

 \end{proof}

\begin{lem}
\label{slem2a} If \;$c_1= 0$\; and \;$c_2\neq 0$\; and $c_{4+k}=0$ for $k\in \{1,...,\mu-5\}$ then the algebraic restriction $\mathcal{A}=[ \sum_{l=1}^{\mu}c_{l} \theta_{l}]_{S_{\mu}}$ can be reduced by a symmetry of $S_{\mu}$ to an algebraic restriction $[\theta_2+ \widetilde{c}_3\theta_3+ \widetilde{c}_{\mu} \theta_{\mu}]_{S_{\mu}}$ where $\widetilde{c}_3 \widetilde{c}_{\mu}=0$ .
\end{lem}

\begin{proof}[Proof of Lemma \ref{slem2a}]
We use analogical methods as in the proof of the previous lemma. Now
$\mathcal{A}= [c_2 \theta_2+c_3 \theta_3+c_4 \theta_4+ c_{\mu}\theta_{\mu}]_{S_{\mu}}$. \par \noindent  When $c_3\ne 0$ let $\mathcal{B}_t=[c_2 \theta_2+c_3 \theta_3+(1-t)c_4 \theta_4+(1-t)c_{\mu} \theta_{\mu}]_{S_{\mu}}$ \; for $t \in[0;1]$. Then $\mathcal{B}_0=\mathcal{A}$\; and
\;$\mathcal{B}_1=[c_2 \theta_2+c_3 \theta_3]_{S_{\mu}}$.
 We prove that there exists a family $\Phi_t \in Symm(S_{\mu}),\;t\in [0;1]$ such that
 \begin{equation}
\label{proofslem21a}   \Phi_t^*\mathcal{B}_t=\mathcal{B}_0,\;\Phi_0=id.
\end{equation}
Let $V_t$ be a vector field defined by $\frac{d \Phi_t}{dt}=V_t(\Phi_t)$. Then differentiating (\ref{proofslem21a}) we obtain
  \begin{equation}
\label{proofslem22a}   \mathcal L_{V_t} \mathcal{B}_t=[c_4 \theta_4+{c}_{\mu} \theta_{\mu}]_{S_{\mu}}.
\end{equation}
We are looking for $V_t$ in the form $V_t=\sum_{k=1}^{3}b_k X_k$   where $b_k\in \mathbb{R}$.  Then by Proposition  \ref{s-infinitesimal} equation (\ref{proofslem22a})
has a form

\setlength{\arraycolsep}{1mm}
\begin{small}
\begin{equation}  \label{proofslem23a}
\left[ \begin{array}{ccc}
\!-(r\!+\!2)c_2 & 0 & rc_3  \\
0 & -\frac{3r^2}{2}c_3 & 0
\end{array} \right]
\left[ \begin{array}{c} b_1 \\ b_2 \\ b_3   \end{array} \right] =
\left[ \begin{array}{c} c_4  \\  c_{\mu}  \end{array}  \right]
\end{equation}
\end{small}

\noindent  If \;$ c_3\ne 0$ we can solve (\ref{proofslem23a}) and $\Phi_t$ may be obtained as a flow of vector field $V_t$.
The family $\Phi_t$ preserves $S_{\mu}$, because $V_t$ is tangent to $S_{\mu}$ and $\Phi_t^*\mathcal{B}_t=\mathcal{A}$.
Using the homotopy arguments we have that $\mathcal{A}$ is diffeomorphic to $ \mathcal{B}_1=[c_2 \theta_2+c_3 \theta_3]_{S_{\mu}}$.
By the condition $c_2\ne 0$ we have a diffeomorphism $\Psi \in Symm(S_{\mu})$ of the form
  \begin{equation}
\label{proofslem24a}
\Psi:\,(x_1,x_2,x_3)\mapsto (c_2^{-\frac{r}{r+2}} x_1,c_2^{-\frac{r}{r+2}} x_2,c_2^{-\frac{2}{r+2}} x_3)
\end{equation}
and we obtain
\[ \Psi^*(\mathcal{B}_1)=[\theta_2+ c_3 c_2^{-\frac{2r}{r+2}}\theta_3]_{S_{\mu}} =
 [ \theta_2+ \widetilde{c}_3\theta_3]_{S_{\mu}}.\]

 \noindent  In the case \;$ c_3= 0$ we take $\mathcal{B}_t=[c_2 \theta_2+(1-t)c_4 \theta_4+c_{\mu} \theta_{\mu}]_{S_{\mu}}$ \; for $t \in[0;1]$ and we
 can solve only the first equation of (\ref{proofslem23a}).
Using the homotopy arguments we have that $\mathcal{A}$ is diffeomorphic to $ \mathcal{B}_1=[c_2 \theta_2+c_{\mu} \theta_{\mu}]_{S_{\mu}}$.
Using the diffeomorphism (\ref{proofslem24a}) we obtain
\[ \Psi^*(\mathcal{B}_1)=[\theta_2+ c_{\mu} c_2^{-\frac{3r}{r+2}}\theta_{\mu}]_{S_{\mu}} =
 [ \theta_2+ \widetilde{c}_{\mu}\theta_{\mu}]_{S_{\mu}}.\]

 \end{proof}

\begin{lem}
\label{slem3} If \;$c_1= 0$\;and $c_2 = 0$ and $c_3 c_{4+k} \ne 0$ and $c_l=0$ for $5\leq l< 4+k$  then the algebraic restriction $\mathcal{A}=[ \sum_{l=1}^{\mu}c_{l} \theta_{l}]_{S_{\mu}}$ can be reduced by a symmetry of $S_{\mu}$ to an algebraic restriction $[\theta_3+ \widetilde{c}_{4+k}\theta_{4+k}+ +\widetilde{c}_{5+k} \theta_{5+k}]_{S_{\mu}}$.
\end{lem}

\begin{proof}[Proof of Lemma \ref{slem3}]
\par \noindent If $c_1=0, c_2=0$ and $c_3 \neq 0$ and $c_{4+k} \ne 0$ and $c_l=0$ for $5\leq l< 4+k$ then
$\mathcal{A}= [c_3 \theta_3+c_4 \theta_4+\sum_{l=4+k}^{\mu}c_{l} \theta_{l}]_{S_{\mu}}$. \par \noindent  Let
$\mathcal{B}_t=[c_3 \theta_3+(1-t)c_4 \theta_4+c_{4+k} \theta_{4+k}+  \sum_{l=5+k}^{\mu}\widetilde{c}_{l}(t) \theta_{l}]_{S_{\mu}}$
\;for $t \in[0;1]$ where $\widetilde{c}_{l}(t)$ are smooth functions $\widetilde{c}_{l}(t):[0;1]\rightarrow \mathbb{R}$ such that $\widetilde{c}_{l}(0)=c_{l}$. Then $\mathcal{B}_0=\mathcal{A}$\; and
\;$\mathcal{B}_1=[c_3 \theta_3+c_{4+k} \theta_{4+k}+\sum_{l=5+k}^{\mu}\widetilde{c}_{l}(1) \theta_{l}]_{S_{\mu}}$.

Let $\Phi_t ,\;t\in [0;1]$, be the flow of the vector field $V=\frac{c_4}{rc_3}X_3$. We show that there exist functions $\widetilde{c}_l$ such that
\begin{equation}
\label{proofslem31a}   \Phi_t^*\mathcal{B}_t=\mathcal{B}_0,\;\Phi_0=id.
\end{equation}
Then differentiating (\ref{proofslem31a}) we obtain
  \begin{equation}
\label{proofslem32a}   \mathcal L_{V} \mathcal{B}_t=[c_4 \theta_4 -\sum_{l=5+k}^{\mu}\frac{d\widetilde{c}_{l}}{dt} \theta_{l}]_{S_{\mu}}.
\end{equation}
We can find the $\widetilde{c}_l$ as solutions of the system of first order linear ODEs defined by (\ref{proofslem32a}) with the initial data $\widetilde{c}_{l}(0)=c_{l}$ for $l=5+k,\ldots,\mu$. This implies that $\mathcal{B}_0=\mathcal{A}$ and $\mathcal{B}_1=[c_3 \theta_3+c_{4+k} \theta_{4+k}+\sum_{l=5+k}^{\mu}\widetilde{c}_{l}(1) \theta_{l}]_{S_{\mu}}$ are diffeomorphic.
Denote $\hat{c}_l=\widetilde{c}_l(1)$ for $l=5+k,\ldots,\mu$.

Next let $\mathcal{C}_t=[c_3 \theta_3+c_{4+k} \theta_{4+k}+\hat{c}_{5+k} \theta_{5+k}+(1-t)\sum_{l=6+k}^{\mu}\hat{c}_{l} \theta_{l}]_{S_{\mu}}$
\;for $t \in[0;1]$.

Then $\mathcal{C}_0=\mathcal{B}_1$\; and
\;$\mathcal{C}_1=[c_3 \theta_3+c_{4+k} \theta_{4+k}+\hat{c}_{5+k} \theta_{5+k}]_{S_{\mu}}$.

 We prove that there exists a family $\Upsilon_t \in Symm(S_{\mu}),\;t\in [0;1]$ such that
 \begin{equation}
\label{proofslem31}   \Upsilon_t^*\mathcal{C}_t=\mathcal{C}_0,\;\Upsilon_0=id.
\end{equation}
Let $V_t$ be a vector field defined by $\frac{d \Upsilon_t}{dt}=V_t(\Upsilon_t)$. Then differentiating (\ref{proofslem31}) we obtain
  \begin{equation}
\label{proofslem32}   \mathcal L_{V_t} \mathcal{B}_t=[\sum_{l=6+k}^{\mu}\hat{c}_{l} \theta_{l}]_{S_{\mu}}.
\end{equation}
We are looking for $V_t$ in the form $V_t=\sum_{k=4}^{\mu-2}b_k(t) X_k$   where  $b_k(t)$ are smooth functions $b_k:[0;1]\rightarrow \mathbb{R}$ for $k=4,\ldots,\mu-2$.  Then by Proposition  \ref{s-infinitesimal} equation (\ref{proofslem32})
has a form

\setlength{\arraycolsep}{0.5mm}
\begin{small}
\begin{equation}  \label{proofslem33}
\left[ \begin{array}{cccccc}
 (r\!+\!2k+6)c_{k\!+\!4} & 0  & 0 \  & \cdots & \cdots & 0 \\
  (r\!+\!2k+8)\hat{c}_{k\!+\!5} & (r\!+\!2k+8)c_{k\!+\!4}  & 0 \  & \cdots & \cdots & 0 \\
  \vdots & \vdots & \ddots & 0 &\cdots & 0 \\
  3r(1-t)\hat{c}_{\mu-1} & 3r\hat{c}_{\mu-2}(1-t) &\cdots & 3rc_{k\!+\!4} & 0 \cdots & 0
\end{array} \right]
\left[ \begin{array}{c} b_4(t)  \\ b_5(t) \\ \vdots  \\ b_{\mu-2}(t)  \end{array} \right] =
\left[ \begin{array}{c} 
\hat{c}_{k+6} \\ \hat{c}_{k+7}\\ \vdots \\  \hat{c}_{\mu}  \end{array}  \right]
\end{equation}
\end{small}

\noindent If \;$c_{4+k}\ne 0$ we can solve (\ref{proofslem33}) and $\Upsilon_t$ may be obtained as a flow of vector field $V_t$.
The family $\Upsilon_t$ preserves $S_{\mu}$, because $V_t$ is tangent to $S_{\mu}$ and $\Upsilon_t^*\mathcal{C}_t=\mathcal{C}_0=\mathcal{B}_1$.
Using the homotopy arguments we have that $\mathcal{A}$ is diffeomorphic to $ \mathcal{B}_1$ and $ \mathcal{B}_1$ is diffeomorphic to $ \mathcal{C}_1$.
By the condition $c_3\ne 0$ we have a diffeomorphism $\Psi \in Symm(S_{\mu})$ of the form
  \begin{equation}
\label{proofslem34}
\Psi:\,(x_1,x_2,x_3)\mapsto (|c_3|^{-\frac{1}{2}} x_1,|c_3|^{-\frac{1}{2}} x_2,|c_3|^{-\frac{1}{r}} x_3)
\end{equation}
and we obtain
\[ \Psi^*(\mathcal{C}_1)=[ \frac{c_3}{|c_3|}\theta_3+\widetilde{c}_{4+k} \theta_{4+k}+\widetilde{c}_{5+k} \theta_{5+k}]_{S_{\mu}} =
 [ sgn(c_3)\theta_3+\widetilde{c}_{4+k} \theta_{4+k}+\widetilde{c}_{5+k} \theta_{5+k}]_{S_{\mu}}.\]

 \noindent By the following symmetry of $S_{\mu}$: $(x_1,x_2,x_3)\mapsto (-x_1, x_2, x_3)$, we have that
 $[ -\theta_3+\widetilde{c}_{4+k} \theta_{4+k}+\widetilde{c}_{5+k} \theta_{5+k}]_{S_{\mu}}$ is diffeomorphic to $[ \theta_3-\widetilde{c}_{4+k} \theta_{4+k}-\widetilde{c}_{5+k} \theta_{5+k}]_{S_{\mu}}$.

 \end{proof}

 \begin{lem}
\label{slem3a} If \;$c_1= 0$\;and $c_2 = 0$ and $c_3\ne 0$ and $c_l=0$ for $5\leq l< \mu-1$  then the algebraic restriction $\mathcal{A}=[ \sum_{l=1}^{\mu}c_{l} \theta_{l}]_{S_{\mu}}$ can be reduced by a symmetry of $S_{\mu}$ to an algebraic restriction $[\theta_3+ \widetilde{c}_{\mu-1} \theta_{\mu-1}]_{S_{\mu}}$.
\end{lem}

\begin{proof}[Proof of Lemma \ref{slem3a}]

The prove of this lemma is very similar to previous case. It suffices to notice that if $c_3\ne 0$ we can solve the following equation
\setlength{\arraycolsep}{1mm}
\begin{small}
\begin{equation}  \label{proofslem33a}
\left[ \begin{array}{cc}
 0 & rc_3  \\
 -\frac{3r^2}{2}c_3 & 3rc_{\mu-1}
\end{array} \right]
\left[ \begin{array}{c}  b_2 \\ b_3   \end{array} \right] =
\left[ \begin{array}{c} c_4  \\  c_{\mu}  \end{array}  \right]
\end{equation}
\end{small}
\end{proof}

\begin{lem}
\label{slem4} If \;$c_1=c_2=c_3= 0$\; and $c_{4+k} \ne 0$ and $c_l=0$ for $5\leq l< 4+k$  then the algebraic restriction $\mathcal{A}=[ \sum_{l=1}^{\mu}c_{l} \theta_{l}]_{S_{\mu}}$ can be reduced by a symmetry of $S_{\mu}$ to an algebraic restriction $[c_4\theta_4+ c_{4+k} \theta_{4+k}]_{S_{\mu}}$.
\end{lem}

\begin{proof}[Proof of Lemma \ref{slem4}]
We use similar methods as above to prove that lemma. In this case
$\mathcal{A}\!= [c_4 \theta_4+\sum_{l=4+k}^{\mu}c_{l} \theta_{l}]_{S_{\mu}}$.   Let $\mathcal{B}_t\!=[c_4 \theta_4+c_{4+k} \theta_{4+k}\!+(1\!-\!t)\sum_{l=5+k}^{\mu}c_{l} \theta_{l}]_{S_{\mu}}$ \; for $t \in[0;1]$. Then $\mathcal{B}_0=\mathcal{A}$\; and
\;$\mathcal{B}_1=[c_4 \theta_4+c_{4+k} \theta_{4+k}]_{S_{\mu}}$.
 We prove that there exists a family $\Phi_t \in Symm(S_{\mu}),\;t\in [0;1]$ such that
 \begin{equation}
\label{proofslem41}   \Phi_t^*\mathcal{B}_t=\mathcal{B}_0,\;\Phi_0=id.
\end{equation}
Let $V_t$ be a vector field defined by $\frac{d \Phi_t}{dt}=V_t(\Phi_t)$. Then differentiating (\ref{proofslem41}) we obtain
  \begin{equation}
\label{proofslem42}   \mathcal L_{V_t} \mathcal{B}_t=[\sum_{l=5+k}^{\mu}c_{l} \theta_{l}]_{S_{\mu}}.
\end{equation}
We are looking for $V_t$ in the form $V_t=\sum_{k=3}^{\mu-2}b_k(t) X_k$   where $b_k(t)$ are smooth functions $b_k:[0;1]\rightarrow \mathbb{R}$ for $k=3,\ldots,\mu-2$.  Then by Proposition  \ref{s-infinitesimal} equation (\ref{proofslem42})
has a form

\setlength{\arraycolsep}{0.3mm}
\begin{small}
\begin{equation}  \label{proofslem43}
\left[ \begin{array}{ccccccc}

 (r\!+\!2k+4)c_{k\!+\!4} & & 0 & 0 \  & \cdots & \cdots & 0 \\
  (r\!+\!2k\!+\!6)c_{k\!+\!5}(1-t) & & (r\!+\!2k\!+\!6)c_{k\!+\!4} & 0 \  & \cdots & \cdots & 0 \\
 \vdots \;\; \;\ddots &  &  \ddots & \ddots & 0 &\cdots & 0 \\
 3rc_{\mu-1}(1-t) & & 3rc_{\mu-2}(1-t) & \cdots & 3rc_{k\!+\!4} & 0 \cdots & 0
\end{array} \right]
\left[ \begin{array}{c}  b_3(t) \\ \vdots  \\ b_{\mu-2}(t)  \end{array} \right] =
\left[ \begin{array}{c}  c_{k+5} \\ c_{k+6} \\ \vdots \\  c_{\mu}  \end{array}  \right]
\end{equation}
\end{small}

\noindent If \;$ c_{4+k}\ne 0$ we can solve (\ref{proofslem43}) and $\Phi_t$ may be obtained as a flow of vector field $V_t$.
The family $\Phi_t$ preserves $S_{\mu}$, because $V_t$ is tangent to $S_{\mu}$ and $\Phi_t^*\mathcal{B}_t=\mathcal{A}$.
Using the homotopy arguments we have that $\mathcal{A}$ is diffeomorphic to $ \mathcal{B}_1=[c_4 \theta_4+c_{4+k} \theta_{4+k}]_{S_{\mu}}$.

\bigskip

When $c_4\ne 0$ we have a diffeomorphism $\Psi \in Symm(S_{\mu})$ of the form
  \begin{equation}
\label{proofslem44}
\Psi:\,(x_1,x_2,x_3)\mapsto (|c_4|^{-\frac{r}{2r+2}} x_1,|c_4|^{-\frac{r}{2r+2}} x_2,|c_4|^{-\frac{2}{2r+2}} x_3)
\end{equation}
and we obtain
\[ \Psi^*(\mathcal{B}_1)=[sgn(c_4)\theta_4+c_{4+k} |c_4|^{-(\frac{2k+r+2}{2r+2})} \theta_{4+k}]_{S_{\mu}} =
 [ \pm\theta_4+ \widetilde{c}_{4+k} \theta_{4+k}]_{S_{\mu}}.\]

 \noindent By the following symmetry of $S_{\mu}$: $(x_1,x_2,x_3)\mapsto (-x_1, x_2, x_3)$, we have that
 $[ -\theta_4+\widetilde{c}_{4+k} \theta_{4+k}]_{S_{\mu}}$ is diffeomorphic to $[ \theta_4-\widetilde{c}_{4+k} \theta_{4+k}]_{S_{\mu}}$.

 \bigskip

  When $c_{4+k}\ne0$ then  we may use a diffeomorphism $\Psi_1\in Symm(S_{\mu})$ of the form
  \begin{equation}
\label{proofslem45}
\Psi_1:\,(x_1,x_2,x_3)\mapsto (c_{4+k}^{-\frac{r}{2k+r+2}} x_1,c_{4+k}^{-\frac{r}{2k+r+2}} x_2,c_{4+k}^{-\frac{2}{2k+r+2}} x_3)
\end{equation}
and we obtain
\[ \Psi_1^*(\mathcal{B}_1)=[c_4 c_{4+k}^{-(\frac{2r+2}{2k+r+2})}\theta_4 + \theta_{4+k}]_{S_{\mu}}=[ \widetilde{c}_4\theta_4+\theta_{4+k}]_{S_{\mu}}.\]


 \end{proof}

Statement $(ii)$ of Theorem \ref{klas_s} follows   from Theorem \ref{geom-cond-s}.

 \bigskip

$(iii)$ Now we prove that the parameters $c_i$ are moduli in the normal forms. The proofs are very similar in all cases. We consider as an example
the normal form with two parameters $[\theta_1+c_2\theta_2+c_3\theta_3]_{S_{\mu}}$. From Table \ref{infini s} we see that the tangent space to the orbit
of $[\theta_1+c_2\theta_2+c_3\theta_3]_{S_{\mu}}$ at $[\theta_1+c_2\theta_2+c_3\theta_3]_{S_{\mu}}$ is spanned by the linearly independent algebraic restrictions
$[r\theta_1+rc_2\theta_2+2c_3\theta_3]_{S_{\mu}}$, $[\theta_4]_{S_{\mu}},[\theta_5]_{S_{\mu}}, \ldots, [\theta_{\mu}]_{S_{\mu}}.$ Hence the algebraic restrictions
$[\theta_2]_{S_{\mu}}$ and $[\theta_3]_{S_{\mu}}$ do not belong to it. Therefore the parameters $c_2$ and $c_3$ are independent moduli in the normal form
$[\theta_1+c_2\theta_2+c_3\theta_3]_{S_{\mu}}$.

\medskip

Statement $(iv)$ of Theorem \ref{klas_s} follows from conditions
in the proof of part $(i)$. 

\end{proof}

\bibliographystyle{amsalpha}

\end{document}